\documentclass[10pt]{amsart}
\usepackage{amssymb,amsbsy,amsmath,amsfonts,amssymb,amscd}
\usepackage{latexsym,euscript,exscale}

\title{Banach spaces without minimal subspaces - Examples}
\author {Valentin Ferenczi and Christian Rosendal}
\date {May 2009}
\newcommand {\N}{\mathbb N}
\newcommand {\Q}{\mathbb Q}

\renewcommand{\leq}{\ensuremath{\leqslant}}
\renewcommand{\geq}{\ensuremath{\geqslant}}

\newcommand{\norm}[1]{\lVert#1\rVert}

\newcommand{\til}{\rightarrow}

\newcommand {\Del}{ \; \Big| \;}
\newcommand {\del}{ \; \big| \;}

\newcommand {\ku} {\mathcal}

\newcommand{\inv}{^{-1}}

\numberwithin{equation}{section}

\newtheorem{thm}{Theorem}[section]
\newtheorem{cor}[thm]{Corollary}
\newtheorem{lemme}[thm]{Lemma}
\newtheorem{prop} [thm] {Proposition}
\newtheorem{defi} [thm] {Definition}

\newtheorem{prob}[thm]{Problem}

\begin{document}
\thanks{The first author acknowledges the support of FAPESP grant 2008/11471-6 and  the second author the support of NSF grant DMS 0556368}


\begin{abstract}  We analyse several examples of separable Banach spaces, some of them new,  and relate them to
  several dichotomies obtained in \cite{FR4}, by
classifying them according to which side of the dichotomies they fall.
\end{abstract}

 \maketitle

\tableofcontents

\section{Introduction}\label{intro}

In this article we give several new examples of Banach spaces, corresponding
to different classes of a list  defined in \cite{FR4}. This paper may be seen as
a more empirical continuation of \cite{FR4} in which our stress is on the
study of examples for the new classes of
Banach spaces considered in that work.

\subsection{Gowers' list of inevitable classes}

In the paper \cite{g:dicho}, W.T. Gowers had defined a program of isomorphic
classification of Banach spaces. The aim of this program is   a {\em loose
classification of Banach spaces up to subspaces}, by producing  a list of classes of Banach spaces such that:

(a)  if a space belongs to a class,
then every subspace belongs to the same class, or maybe, in the case when the
properties defining the class depend on a basis of the space, every block
subspace belongs to the same class,

(b) the classes are {\em inevitable}, i.e., every Banach space contains a subspace in one of the classes,

(c) any two classes in the list are disjoint,

(d) belonging to one class gives a lot of information about operators that may be defined on the space or on its subspaces.

\

We shall refer to such a list as a {\em list of  inevitable classes of
  Gowers}.
For the motivation of Gowers' program as well as the relation of this program
  to classical problems in Banach space theory we refer to \cite{FR4}.
Let us just say that the class of spaces $c_0$ and $\ell_p$ is seen as the
  nicest or most regular class, and
  so, the objective of Gowers' program really is the classification of those spaces (such as Tsirelson's
  space $T$) which do not contain a copy of $c_0$ or $\ell_p$. Actually, in \cite{FR4}, mainly spaces without {\em minimal subspaces} are
  classified, and so in this article, we shall consider various examples of
  Banach spaces without minimal subspaces. 
We shall first  give a  summary of the classification obtained in \cite{FR4}
and of the results that led to that classification.

After the construction by Gowers and Maurey of a {\em hereditarily
indecomposable} (or HI) space $GM$, i.e.,  a space such that no subspace may be
written as the direct sum of infinite dimensional subspaces \cite{GM}, Gowers
proved that every Banach space contains either an HI subspace or a subspace
with an unconditional basis \cite{g:hi}.
This dichotomy is called {\em first dichotomy} of
Gowers in \cite{FR4}.
 These were the first two examples of
inevitable classes.  He then 
refined the list by proving a {\em second dichotomy}: any Banach space contains a subspace with a
basis such that either no two disjointly supported block subspaces are
isomorphic, or such that any two subspaces have further
subspaces which are isomorphic. He called the second property {\em quasi
minimality}. Finally,
H. Rosenthal had defined  a space to be {\em minimal} if it embeds into any of its subspaces. A
quasi minimal space which does not contain a minimal subspace is called {\em
strictly quasi minimal}, so Gowers again divided the class of quasi minimal
spaces into the class of strictly quasi minimal spaces and the class of minimal
spaces.

Gowers therefore produced a list of  four inevitable classes of Banach spaces,
corresponding to classical examples, or more recent couterexamples to
classical questions:
HI spaces, such as $GM$; spaces with bases such that no disjointly supported
subspaces are isomorphic, such as the
couterexample $G_u$ of Gowers to the hyperplane's problem of Banach \cite{g:hyperplanes};
strictly quasi minimal spaces with an unconditional basis, such as Tsirelson's space  $T$ \cite{tsi} ; and
finally, minimal spaces, such as $c_0$ or $\ell_p$, but also $T^*$,
Schlumprecht's space $S$ \cite{S1}, or as proved recently in \cite{MP}, its dual $S^*$.

\

\subsection{The three new dichotomies}
In \cite{FR4} three  dichotomies for Banach spaces  were
obtained. The first one of these new dichotomies,  
 the {\em third dichotomy}, concerns the property of minimality defined by Rosenthal. Recall that a
 Banach space is minimal if it embeds into any of its infinite dimensional
 subspaces. On the other hand, a space $Y$ is {\em
tight} in a basic sequence $(e_i)$  if there is a sequence of successive
subsets $I_0<I_1<I_2<\ldots$ of $\N$, such that for all infinite subsets
$A\subseteq \N$, we have
$$
Y\not\sqsubseteq [e_n\del n\notin \bigcup_{i\in A}I_i].
$$

A {\em tight basis} is a basis such that every subspace is tight
in it, and a {\em tight space} is a space with a tight basis \cite{FR4}.

The subsets $I_n$ may clearly be chosen to be intervals or even to form a partition of $\N$. However it is convenient not to require this condition in the definition, in view of forthcoming special cases of tightness.

 It is  observed in \cite{FR4} that the tightness property is hereditary,
 incompatible with minimality, and it is proved that: 

\begin{thm}[3rd dichotomy, Ferenczi-Rosendal 2007]\label{main}
Let $E$ be a Banach space without minimal subspaces. Then $E$ has a tight subspace.
\end{thm}

Actual examples of tight spaces in \cite{FR4} turn out to satisfy  one of two stronger forms of tightness.
The first was called  {\em tightness by range}. Here the
range, ${\rm range} \ x$, of  a vector $x$ is the smallest interval of integers
containing its support on the given basis, and the range  of a block subspace $[x_n]$ is
$\bigcup_n {\rm range} \ x_n$. A  basis $(e_n)$ is tight by range when for
every block subspace $Y=[y_n]$, the sequence of successive subsets
$I_0<I_1<\ldots$ of $\N$ witnessing the tightness of $Y$ in $(e_n)$ may be
defined by $I_k={\rm range}\ y_k$ for each $k$. This is equivalent to no two
block subspaces with disjoint ranges being comparable, where two spaces are comparable if one embeds into the other.

When  the definition of tightness may be checked with $I_k={\rm supp}\ y_k$
instead of ${\rm range}\ y_k$, then a stronger property is obtained which is called
tightness by support, and is 
equivalent to the property defined by Gowers in the second dichotomy  that no disjointly
supported block subspaces are isomorphic,  Therefore $G_u$ is an example of
space with a basis which is tight by support and therefore by range.

The second kind of tightness was called {\em tightness with constants}. A basis $(e_n)$ is
tight with constants when for for every infinite dimensional space $Y$, the
sequence of successive subsets $I_0<I_1<\ldots$ of $\N$ witnessing the
tightness of $Y$ in $(e_n)$ may be chosen so that
$Y \not\sqsubseteq_K [e_n \del n \notin I_K]$ for each $K$. This is the case for Tsirelson's space $T$ or its $p$-convexified version $T^{(p)}$ \cite{CS}.

As we shall see, one of the aims of this paper is to present various examples of tight spaces
of these two forms.

\

In \cite{FR4} it was proved  that there  are natural dichotomies between each of these strong
forms of tightness and respective weak forms of minimality. For the first
notion, a space $X$ with a
basis $(x_n)$ is said to be {\em subsequentially minimal} if every subspace of  $X$
contains an isomorphic copy of a subsequence of $(x_n)$. Essentially
this notion had  been previously considered by Kutzarova, Leung, Manoussakis
and Tang in the context of modified partially mixed Tsirelson spaces \cite{KLMT}.

\begin{thm}[4th dichotomy, Ferenczi-Rosendal 2007]\label{main2}
Any Banach space $E$ contains a
subspace with a basis that is either tight by range or is subsequentially minimal.
\end{thm}

The second case in Theorem \ref{main2}
may be improved to the following hereditary property of a basis $(x_n)$, that we call {\em
sequential minimality}: $(x_n)$ is quasi minimal and  every block sequence of
$[x_n]$ has a subsequentially minimal block sequence.

There is also a dichotomy concerning tightness with constants. Recall that given two Banach spaces $X$ and $Y$, we say that $X$ is {\em crudely finitely
representable} in $Y$ if there is a constant $K$ such that for any
finite-dimensional subspace $F\subseteq X$ there is an embedding $T\colon F\til
Y$ with constant $K$, i.e., $\norm{T}\cdot\norm{T\inv}\leq K$.
 A space $X$ is said to be {\em locally minimal} if for some constant
$K$, $X$ is $K$-crudely finitely representable in any of its subspaces.
\begin{thm}[5th dichotomy, Ferenczi-Rosendal 2007] \label{main3}
Any Banach space $E$ contains a subspace with a basis that is either tight
with constants or is locally minimal.
\end{thm}

\

Finally there exists a sixth
 dichotomy theorem due to A.
Tcaciuc \cite{T}, stated here in a slightly strengthened form. A space $X$ is 
 {\em uniformly inhomogeneous} when
$$\forall M\geq 1\;  \exists n \in \N\; \forall Y_1,\ldots,Y_{2n} \subseteq X\;
\exists y_i \in\ku S_{Y_i}\;(y_i)_{i=1}^n\not\sim_M(y_i)_{i=n+1}^{2n},$$
where $Y_1,\ldots,Y_{2n}$ are assumed to be infinite-dimensional subspaces of $X$.
On the contrary, a  basis $(e_n)$  is said to be {\em strongly asymptotically
$\ell_p$}, $1 \leq p \leq +\infty$, \cite{DFKO}, if there exists a constant $C$
and a function $f:\N \rightarrow \N$ such that for any $n$, any family of $n$
unit vectors which are disjointly supported in $[e_k \del k \geq f(n)]$ is
$C$-equivalent to the canonical basis of $\ell_p^n$.
Tcaciuc then proves \cite{T} :

\begin{thm}[Tcaciuc's dichotomy, 2005]
Any Banach space contains a subspace with a basis which is either uniformly
inhomogeneous or strongly asymptotically $\ell_p$ for some 
$1 \leq p \leq +\infty$.
\end{thm}

The six dichotomies and the interdependence of the properties involved can be
visualised in the following diagram.

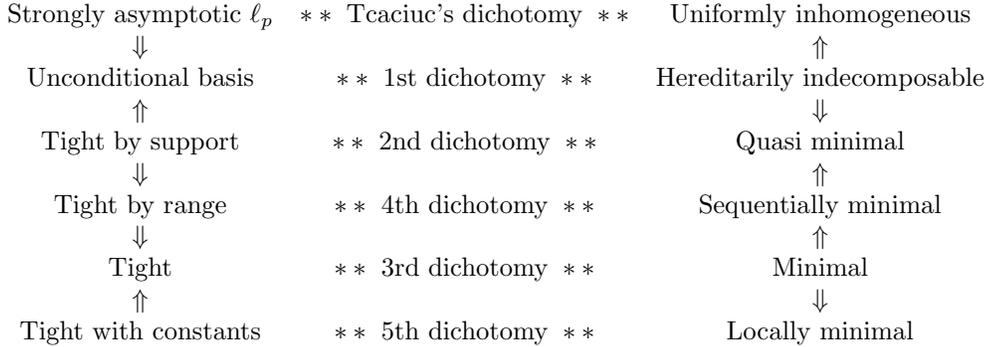
\[
\begin{tabular}{ccc}
Strongly asymptotic $\ell_p$&$**\textrm{ Tcaciuc's dichotomy }**$&
Uniformly inhomogeneous\\

$\Downarrow$&         &$\Uparrow$\\

Unconditional basis&$**\textrm{ 1st dichotomy }**$& Hereditarily indecomposable\\

$\Uparrow$&         &$\Downarrow$\\

Tight by support & $**\textrm{ 2nd dichotomy }**$ & Quasi minimal  \\

$\Downarrow$&&$\Uparrow$\\

Tight by range    &      $**\textrm{ 4th dichotomy }**$          & Sequentially minimal     \\

$\Downarrow$&&$\Uparrow$\\

Tight&                  $**\textrm{ 3rd dichotomy }**$              & Minimal\\

$\Uparrow$&         &$\Downarrow$\\

Tight with constants& $**\textrm{ 5th dichotomy }**$ & Locally minimal\\
\end{tabular}
\]

\

Moreover, 
$${\rm Strongly\ asymptotic\ } \ell_p {\rm\ not\ containing\ } \ell_p. 1 \leq p <+\infty
\Rightarrow {\rm Tight\ with\ constants},$$
and
$${\rm Strongly\ asymptotic\ } \ell_{\infty} 
\Rightarrow {\rm Locally\ minimal}.$$

\

Note that while a basis tight by support must be unconditional, a basis which is tight by range may span a HI space. So tightness by support and tightness by range are two different notions. We would lose this subtle difference if we required the sets $I_n$ to be intervals in the definition of tightness.
Likewise a basis may be tight by range without being (nor containing a basis which is) tight with constants, and tight with constants without being (nor containing a basis which is) tight by range. Actually none of the converses of the implications appearing on the left or the right of the list of  the six dichotomies holds, even if one allows passing to a further subspace. All the claims of this paragraph are easily checked by looking at the list of examples of Theorem \ref{final}, which is the aim of this paper.

The fact that a strongly asymptotically $\ell_p$ space not containing $\ell_p$ must be tight with constants is proved in \cite{FR4} but is essentially due to the authors of \cite{DFKO}, and
 the observation that such bases are unconditional  may also be found in  \cite{DFKO}. The easy fact that HI spaces are uniformly homogeneous (with $n=2$ in the definition) is observed in \cite{FR4}. That HI spaces are quasi-minimal  is due to Gowers \cite{g:dicho}, and that minimal spaces are locally minimal is a consequence of an observation by P. G. Casazza \cite{C} that every minimal space must $K$-embed into all its subspaces for some $K \geq 1$.
The other implications are direct consequences of the definitions, and more explanations and details may be found in \cite{FR4}.

\

\subsection{The list of 19 inevitable classes} Combining the six dichotomies and the relations between them, the following
list of 19 classes of Banach spaces contained in any Banach space is obtained
in \cite{FR4}:

\begin{thm}[Ferenczi - Rosendal 2007]\label{final} Any infinite dimensional Banach space contains a subspace of one of
the types listed in the following chart:
\begin{center}
  \begin{tabular}{|l|l|l|}

    \hline

    Type       & Properties                              & Examples                                  \\

    \hline

    (1a)                     & HI, tight by range and with constants    &   ?\\

    (1b)                   & HI, tight by range, locally minimal      & $G^*$\\

    \hline

    (2) & HI, tight, sequentially minimal     &   ? \\

    \hline

    (3a) & tight by support and with constants, uniformly inhomogeneous & ? \\

 (3b) & tight by support, locally minimal, uniformly inhomogeneous & $G_u^*$ \\

 (3c) & tight by support, strongly asymptotically
                             $\ell_p$, $1 \leq p <\infty$ & $X_u, X_{abr}$ \\

(3d) & tight by support, strongly asymptotically
                             $\ell_{\infty}$ & $X_u^*$  \\

    \hline

(4) & unconditional basis, quasi minimal, tight by range & ? \\

\hline

(5a) & unconditional basis, tight with constants, sequentially minimal,  & ? \\

     & uniformly inhomogeneous &    \\

(5b) & unconditional basis, tight, sequentially and locally minimal,
& ? \\

     & uniformly inhomogeneous &    \\

(5c) & tight with constants, sequentially minimal, & $T$, $T^{(p)}$ \\

     &  strongly asymptotically $\ell_p$, $1 \leq p<\infty$ &   \\

(5d) & tight, sequentially minimal, strongly asymptotically $\ell_{\infty}$ & ?\\

\hline

(6a) & unconditional basis, minimal, uniformly inhomogeneous & $S,S^*$ \\

(6b) & minimal, reflexive, strongly asymptotically $\ell_{\infty}$ & $T^*$\\

(6c) & isomorphic to $c_0$ or $l_p$, $1 \leq p<\infty$ & $c_0$, $\ell_p$\\

\hline
  \end{tabular}

\end{center}
\end{thm}

\

The class of type (2) spaces may be divided into two subclasses, using the 5th dichotomy, and the class of type (4) into four, using the 5th and the 6th dichotomy, giving a total of 19 inevitable classes. Since we know of no example of a type (2) or type (4) space to begin with, we do not write down the list of  possible subclasses of these two classes, leaving this as an exercise to the interested reader. 

Note that the tightness property may be used to obtain lower bounds of complexity for the relation of isomorphism between subspaces of a given Banach space. This was initiated by B. Bossard \cite{Bo} who used Gowers' space $G_u$ and its tightness by support. Other results in this direction may be found in \cite{FR4}. We also refer to  \cite{ergodic} for a more introductory work to this question.

\

In \cite{FR4} the existence of $X_u$ and the properties of  $S$, $G$, $G_u$ and $X_u$ which appear in the chart  and are mentioned without proof. It is the
main objective of this paper to prove the results about the spaces which
appear in the above chart.


\

So in what follows  various (and for some of
them new) examples of ``pure'' tight spaces are analysed combining
some of the properties of tightness or minimality associated to each dichotomy.
We shall provide several examples of tight spaces from the two main families of exotic
Banach spaces: spaces of the type of Gowers and Maurey \cite{GM} and spaces of the type of
Argyros and Deliyanni \cite{AD}. Recall that  both types of spaces are defined
using a coding procedure to
``conditionalise'' the norm of some ground space  defined by induction.
In spaces of the type of Gowers and Maurey,  the ground space is the space $S$
of Schlumprecht, and in spaces of the type of Argyros and Deliyanni, it is
a mixed (in further versions  modified or partly modified)  Tsirelson space associated to the  sequence of
Schreier  families.
The space $S$ is far from being asymptotic $\ell_p$ and is actually uniformly
inhomogeneous, and this is the case for our examples of the type
of Gowers-Maurey as well. On the other hand, we  use  a space in the second
family, inspired by an example of Argyros, Deliyanni, Kutzarova and
Manoussakis \cite{ADKM}, to produce strongly asymptotically $\ell_1$ and
$\ell_{\infty}$  examples with strong tightness properties.

\section{Tight  unconditional spaces of the type of Gowers and Maurey}

In this section we prove that the dual of the type (3) space $G_u$ constructed
by Gowers in \cite{g:hyperplanes} is locally minimal  of type (3),
 that Gowers' hereditarily indecomposable and asymptotically unconditional
 space $G$ defined in \cite{g:asymptotic} is of type (1), and that its dual
$G^*$ is locally minimal of type (1). These spaces are
 natural variations on Gowers and Maurey's space $GM$, and so familiarity with
 that construction will be assumed: we shall not redefine the now classical
 notation relative to $GM$, such as the sets of integers $K$ and $L$,  rapidly increasing sequences (or R.I.S.), the
 set ${\bf Q}$ of functionals, special functionals, etc., instead we shall try to give details on the
 parts in which $G_u$ or $G$ differ from $GM$.

The idea of the proofs is similar to \cite{g:hyperplanes}. The HI property for
Gowers-Maurey's spaces is obtained as follows. Some vector $x$ is constructed
such that $\norm{x}$ is large, but so that if $x'$ is obtained from $x$ by
changing signs of the components of $x$, then $x^*(x')$ is small for any
norming functional $x^*$, and so $\norm{x'}$ is small. The upper bound for
$x^*(x')$ is obtained by a combination of unconditional estimates (not
depending on the signs)  and of conditional estimates (i.e., based on the fact
that $|\sum_{i=1}^n \epsilon_i|$ is much smaller than $n$ if
$\epsilon_i=(-1)^i$ for all $i$).

For our examples we shall need to prove that some operator $T$ is
unbounded. Thus we shall construct a vector $x$ such that say $Tx$ has large
norm, and such that $x^*(x)$ is small for any norming $x^*$. The upper bound
for $x^*(x)$ will be obtained by the same unconditional estimates as in the HI
case, while conditional estimates will be trivial due to disjointness of
 supports of the corresponding component vectors and functionals.
The method will be  similar for the dual spaces.

\

Recall that if $X$ is a space with a bimonotone basis, an
$\ell_{1+}^n$-average with constant $1+\epsilon$ is a
normalised vector of the form $\sum_{i=1}^n x_i$, where $x_1<\cdots<x_n$ and
$\norm{x_i} \leq \frac{1+\epsilon}{n}$ for all $i$.
An $\ell_{\infty+}^n$-average with constant $1+\epsilon$  is a normalised vector of the form
 $\sum_{i=1}^n x_i$, where $x_1<\cdots<x_n$ and
$\norm{x_i} \geq \frac{1}{1+\epsilon}$ for all $i$.
An $\ell_{1+}^n$-vector (resp. $\ell_{\infty+}^n$-vector) is a non zero
multiple of an $\ell_{1+}^n$-average (resp. $\ell_{\infty+}^n$-average).
The function $f$ is defined by $f(n)=\log_2(n+1)$. The space $X$ is said to satisfy a lower $f$-estimate if for any
$x_1<\cdots<x_n$,
$$\frac{1}{f(n)}\sum_{i=1}^n \norm{x_i} \leq \norm{\sum_{i=1}^n x_i}.$$

\begin{lemme}\label{linftyn}
Let $X$ be a reflexive space with a bimonotone basis and satisfying a lower $f$-estimate.
Let $(y_k^*)$ be a normalised  block sequence of $X^*$, $n \in \N$, $\epsilon,\alpha>0$.
 Then there
exists a constant $N(n,\epsilon)$, successive subsets $F_i$ of
 $[1,N(n,\epsilon)]$, $1 \leq i \leq n$, and $\lambda>0$ such that if
$x_i^*:=\lambda \sum_{k \in F_i} y_k^*$ for all $i$, then   $x^*=\sum_{i=1}^n
x_i^*$ \ is an $\ell_{\infty +}^n$- average with constant $1+\epsilon$.
 Furthermore, if for each $i$, $x_i$ is such that $\norm{x_i} \leq 1$,
${\rm range}\;x_i \subseteq {\rm range}\;x_i^*$ and $x_i^*(x_i) \geq \alpha \norm{x_i^*}$, then
 $x=\sum_{i=1}^n x_i$ is an $\ell_{1+}^n$-vector with constant
 $\frac{1+\epsilon}{\alpha}$ such that $x^*(x) \geq \frac{\alpha}{1+\epsilon}\norm{x}$.
\end{lemme}

\begin{proof}
Since $X$ satisfies a lower $f$-estimate, it follows by duality that any
sequence of successive functionals $x_1^*<\cdots<x_n^*$ in $G_u^*$ satisfies the following
upper estimate:
$$1 \leq \norm{\sum_{i=1}^n x_i^*} \leq f(n) \max_{1 \leq i \leq n}\norm{x_i^*}.$$
Let $N=n^k$ where $k$ is such that $(1+\epsilon)^k>
  f(n^k)$. Assume towards a contradiction that the result is false for
  $N(n,\epsilon)=N$, then
$$y^*=(y_1^*+\ldots+y_{n^{k-1}}^*)+\ldots+(y_{(n-1)n^{k-1}+1}^*+\ldots+y_{n^k}^*)$$
is not an $\ell_{\infty +}^n$-vector with constant $1+\epsilon$,  and therefore, for some $i$,
$$\norm{y_{i n^{k-1}+1}^*+\ldots+y_{(i+1)n^{k-1}}^*} \leq \frac{1}{1+\epsilon}\norm{y^*}.$$
Applying the same reasoning to the above sum instead of $y^*$, we obtain,
for some $j$,
$$\norm{y_{j n^{k-2}+1}^*+\ldots+y_{(j+1)n^{k-2}}^*} \leq \frac{1}{(1+\epsilon)^2}\norm{y^*}.$$
By induction we obtain that
$$1 \leq \frac{1}{(1+\epsilon)^k} \norm{y^*}
\leq \frac{1}{(1+\epsilon)^k} f(n^k),$$
a contradiction.

Let therefore $x^*$ be such an $\ell_{\infty +}^n$-average with constant
$1+\epsilon$ of the form $\sum_i x_i^*$.
Let  for each $i$, $x_i$ be  such that $\norm{x_i} \leq 1$,
${\rm range}\ x_i \subseteq {\rm range}\ x_i^*$ and $x_i^*(x_i) \geq \alpha \norm{x_i^*}$.
 Then
$$\norm{\sum_i x_i} \geq x^*(\sum_i x_i) \geq \alpha (\sum_i \|x_i^*\|) \geq
\frac{\alpha n}{1+\epsilon},$$ and in particular for each $i$,
$$\norm{x_i} \leq  1 \leq \frac{1+\epsilon}{\alpha n}\norm{\sum_i x_i},$$
so $\sum_i x_i$ is a $\ell_{1+}^n$-vector with constant $\frac{1+\epsilon}{\alpha}$. We also obtain that
$$x^*(\sum_i x_i) \geq
\frac{\alpha n}{1+\epsilon} \geq \frac{\alpha}{1+\epsilon}\norm{\sum_i x_i},$$
as required.
\end{proof}

The following lemma is fundamental and therefore worth stating explicitly. It
appears for example as Lemma 4 in \cite{g:asymptotic}. Recall that an
$(M,g)$-form is a functional of the form $g(M)^{-1}(x_1^*+\ldots+x_M^*)$, with
$x_1^*<\cdots<x_M^*$ of norm at most $1$.

\begin{lemme}[Lemma 4 in \cite{g:asymptotic}]\label{fundamental}
 Let $f,g\in\mathcal F$ with $g\geq\sqrt f$, let $X$ be a space with a bimonotone basis
satisfying a lower $f$-estimate, let $\epsilon>0$ and $\epsilon'=\min\{\epsilon,1\}$, let $x_1,\ldots,x_N$ be a R.I.S. in X
for $f$ with constant $1+\epsilon$ and let $x=\sum_{i=1}^Nx_i$. Suppose that
$$\norm{Ex}\leq\sup\Bigl\{|x^*(Ex)|:M\geq 2, x^*\ \hbox{is an $(M,g)$-form}
\Bigr\}$$
for every interval $E$ such that $\norm{Ex}\ge 1/3$. Then $\norm
x\leq(1+\epsilon+\epsilon')Ng(N)^{-1}$.
 \end{lemme}

\subsection{A locally minimal space tight by support}
Let $G_u$ be the space defined in \cite{g:hyperplanes}. This space has a suppression
unconditional basis,  is tight by
support and therefore reflexive, and its norm is given by the following
implicit equation, for all $x \in c_{00}$:

$$\norm{x}=\norm{x}_{c_0}\vee\ \sup\Bigl\{f(n)^{-1}\sum_{i=1}^n\norm{E_i x}\Del 2 \leq n, E_1<\ldots<E_n\Bigr\}$$
$$\vee\ \sup\Bigl\{|x^*(x)|\Del k \in K, x^* \hbox{ special of length } k \Bigr\}$$

where $E_1, \ldots, E_n$ are successive
subsets (not necessarily intervals)  of $\N$.

\begin{prop}\label{gu} The dual $G_u^*$ of $G_u$ is tight by support and locally minimal.
\end{prop}

\begin{proof}

Given $n \in \N$ and  $\epsilon=1/10$ we may by Lemma \ref{linftyn} find some
$N$ such that there exists in the span of any
$x_1^*<\ldots<x_N^*$ an $\ell_{\infty+}^{n}$-average with constant
$1+\epsilon$.
By unconditionality we deduce that any block-subspace of $G_u^*$ contains
$\ell_{\infty}^n$'s uniformly, and therefore $G_u^*$ is locally minimal.

Assume now towards a contradiction that $(x_n^*)$ and $(y_n^*)$ are disjointly
supported and equivalent block sequences in $G_u^*$, and let $T: [x_n^*] \rightarrow [y_n^*]$ be defined by $Tx_n^*=y_n^*$.

We may assume that each $x_n^*$ is an $\ell_{\infty +}^n$-average with
constant $1+\epsilon$. Using Hahn-Banach theorem, the $1$-unconditionality  of the
basis, and Lemma \ref{linftyn}, we may also find for each $n$ an $\ell_{1+}^n$-average $x_n$ with
constant $1+\epsilon$ such that ${\rm supp}\ x_n \subseteq {\rm supp}\ x_n^*$
and
$x_n^*(x_n) \geq 1/2$.
By construction, for each $n$, $Tx_n^*$ is disjointly supported from $[x_k]$,
and up to modifying $T$, we may assume that $Tx_n^*$  is in ${\bf Q}$ and of
norm at most $1$ for each $n$.

 If $z_1,\ldots,z_m$ is a R.I.S. of these ${\ell}_{1+}^n$-averages $x_n$ with constant $1+\epsilon$,
with $m \in [\log N, \exp N]$, $N \in L$, and $z_1^*,\ldots,z_m^*$ are the
functionals associated to $z_1,\ldots,z_m$, then by \cite{g:hyperplanes} Lemma 7, the $(m,f)$-form
$z^*=f(m)^{-1}(z_1^*+\ldots+z_m^*)$ satisfies
$$z^*(z_1+\ldots+z_m) \geq \frac{m}{2f(m)} \geq
\frac{1}{4}\norm{z_1+\ldots+z_m},$$
and furthermore  $Tz^*$ is also an $(m,f)$-form.
Therefore we may build R.I.S.  vectors $z$ with constant $1+\epsilon$ of
arbitrary length $m$  in $[\log N, \exp N]$,  $N \in L$, so that $z$ is
$4^{-1}$-normed by an $(m,f)$-form $z^*$ such that $Tz^*$ is also
an $(m,f)$-form.
We may then consider a  sequence $z_1,\ldots,z_k$ of length $k \in K$ of such
R.I.S. vectors of length $m_i$, and some corresponding $(m_i,f)$-forms $z_1^*,\ldots,z_k^*$ (i.e
$z_i^*$ $4^{-1}$-norms $z_i$ and $Tz_i^*$ is also an $(m_i,f)$-form for all $i$), such
that
$Tz_1^*,\ldots,Tz_k^*$ is a special sequence.
 Then we let
 $z=z_1+\cdots+z_k$ and $z^*=f(k)^{-1/2}(z_1^*+\ldots+z_k^*)$.
Since $Tz^*=f(k)^{-1/2}(Tz_1^*+\ldots+Tz_k^*)$ is a special functional it follows that
$$\norm{Tz^*} \leq 1.$$Our aim is now to show that $\norm{z} \leq 3kf(k)^{-1}$.
It will then follow that
$$\norm{z^*} \geq z^*(z)/\norm{z} \geq  f(k)^{1/2}/12.$$
 Since $k$ was arbitrary in $K$ this will imply that $T^{-1}$ is unbounded and provide the desired contradiction.

The proof is  almost exactly the same as in  \cite{g:hyperplanes}. Let $K_0=K
\setminus \{k\}$ and let $g$ be the corresponding function given by
\cite{g:hyperplanes} Lemma 6. To prove that $\norm{z} \leq3kf(k)^{-1}$ it is
enough by \cite{g:hyperplanes} Lemma 8 and Lemma \ref{fundamental} to prove
that for any interval $E$ such that $\norm{Ez} \geq 1/3$, $Ez$ is normed by
some $(M,g)$-form with $M \geq 2$.

By the discussion in the proof of the main theorem in  \cite{g:hyperplanes},
the only possible norming functionals apart from $(M,g)$-forms are  special
functionals of length $k$. So let $w^*=f(k)^{-1/2}(w_1^*+\cdots+w_k^*)$ be  a
special functional of length $k$, and $E$ be an interval such that $\norm{Ez}
\geq 1/3$. We need to show that $w^*$ does not norm $Ez$.

Let $t$ be minimal such that $w_t^* \neq Tz_t^*$. If $i \neq j$ or $i=j>t$ then
by definition of special sequences there exist $M \neq N \in L$, $\min(M,N)
\geq j_{2k}$, such that $w_i^*$  is an $(M,f)$-form and $z_j$ is an R.I.S.
vector of size $N$ and constant $1+\epsilon$. By \cite{g:hyperplanes} Lemma~8,
$z_j$ is an $\ell_{1+}^{N^{1/10}}$-average with constant $2$. If $M<N$ then
$2M<\log \log \log N$ so, by \cite{g:hyperplanes} Corollary 3, $|w_i^*(Ez_j)|
\leq 6f(M)^{-1}$. If $M>N$ then $\log \log \log M>2N$  so, by
\cite{g:hyperplanes} Lemma 4, $|w_i^*(Ez_j)| \leq 2f(N)/N$. In both cases it
follows that $|w_i^*(Ez_j)| \leq k^{-2}$.

If $i=j=t$ we have $|w_i^*(Ez_j)| \leq 1$. Finally if $i=j<t$ then
$w_i^*=Tz_i^*$. Since $Tz_i^*$ is disjointly supported from $[x_k]$ and
therefore from $z_j$,
it follows simply that $w_i^*(Ez_j)=0$ in that case.

Summing up we have obtained that
$$|w^*(Ez)| \leq f(k)^{-1/2}(k^2. k^{-2}+1)=2f(k)^{-1/2} < 1/3 \leq \norm{Ez}.$$
Therefore $w^*$ does not norm $Ez$ and this finishes the proof.
\end{proof}

\subsection{Uniformly inhomogeneous examples} It may be observed that  $G_u^*$ is uniformly inhomogeneous. We state this in
a general form which  implies the result for $G_u$, Schlumprecht's space $S$
and its dual $S^*$. This is also true for Gowers-Maurey's space $GM$ and its dual $GM^*$, as well as
for $G$ and  $G^*$, where $G$ is the HI asymptotically unconditional space of
Gowers from \cite{g:asymptotic}, which we shall redefine and study later
on. As HI spaces are always uniformly inhomogeneous however, we need to
observe that a slightly stronger result is obtained by the proof of the next
statement to see that Proposition \ref{spaceswithtcaciuc} is not trivial in
the case of $GM$, $G$ or their duals - see the three paragraphs after Proposition \ref{spaceswithtcaciuc}.

\begin{prop}\label{spaceswithtcaciuc}  Let $f \in {\mathcal F}$ and let $X$ be a space with a bimonotone basis
satisfying a lower $f$-estimate.  Let $\epsilon_0=1/10$, and assume that for every
$n \in [\log N, \exp N], N \in L$, $x_1,\ldots,x_n$  a R.I.S. in $X$
 with constant $1+\epsilon_0$ and  $x=\sum_{i=1}^Nx_i$,
$$\norm{Ex}\leq\sup\Bigl\{|x^*(Ex)|:M\geq 2, x^*\ \hbox{is an $(M,f)$-form}
\Bigr\}$$
for every interval $E$ such that $\norm{Ex}\ge 1/3$. Then $X$ and $X^*$
are uniformly inhomogeneous.
\end{prop}

\begin{proof} Given $\epsilon>0$, let $m \in L$ be such that
$f(m) \geq 24\epsilon^{-1}$.
Let $Y_1,\ldots,Y_{2m}$ be arbitrary block subspaces of $X$. By the classical
method for spaces with a lower $f$ estimate, we may find a R.I.S. sequence $y_1<\cdots<y_m$ with constant
$1+\epsilon_0$ with $y_i \in Y_{2i-1}, \forall i$. By Lemma \ref{fundamental},
$$\norm{\sum_{i=1}^m y_i} \leq 2mf(m)^{-1}.$$
Let on the other hand $n \in [m^{10},\exp m]$ and $E_1<\cdots<E_m$ be sets such
that $\bigcup_{j=1}^m E_j=\{1,\ldots,n\}$ and $|E_j|$ is within $1$ of
$\frac{n}{m}$
 for all $j$. We may construct a R.I.S. sequence $x_1,\ldots,x_n$ with
 constant $1+\epsilon_0$ such that $x_i \in Y_{2j}$ whenever $i \in E_j$.

By Lemma \ref{fundamental},
 $$\norm{\sum_{i \in E_j}x_i} \leq
 (1+2\epsilon_0)(\frac{n}{m}+1)f(\frac{n}{m}-1)^{-1} \leq 2nf(n)^{-1} m^{-1}.$$
Let $z_j=\norm{\sum_{i \in E_j}x_i}^{-1}\sum_{i \in E_j}x_i$. Then $z_j \in
 Y_{2j}$ for all $j$ and
$$\norm{\sum_{j=1}^m z_j} \geq f(n)^{-1}\sum_{j=1}^m \big(\norm{\sum_{i \in E_j}x_i}^{-1}\sum_{i \in E_j}\norm{x_i}\big) \geq m/2.$$
Therefore
$$\norm{\sum_{i=1}^m y_i} \leq 4f(m)^{-1}\norm{\sum_{i=1}^m z_i} \leq \epsilon \norm{\sum_{i=1}^m z_i}.$$
Obviously  $(y_{i})_{i=1}^m$ is not $\epsilon^{-1}$-equivalent to
$(z_{i})_{i=1}^m$, and this means that $X$ is uniformly inhomogeneous.

The proof concerning the dual is quite similar and uses the same notation.
Let $Y_{1},\ldots,Y_{2m}$ be arbitrary block subspaces of $X^*$. By Lemma \ref{linftyn}
 we may find a R.I.S. sequence $y_1<\cdots<y_m$ with constant
$1+\epsilon_0$ and functionals  $y_i^* \in Y_{2i-1}$ such that
${\rm range}\ y_i^* \subseteq {\rm range}\ y_i$ and $y_i^*(y_i) \geq 1/2$ for all
$i$. Since
$\norm{\sum_{i=1}^m y_i} \leq 2mf(m)^{-1}$, it follows that
$$\norm{\sum_{i=1}^m y_i^*} \geq  \norm{\sum_{i=1}^m y_i}^{-1}\sum_{i=1}^m
y_i^*(y_i) \geq f(m)/4.$$
 On the other hand we may construct a R.I.S. sequence $x_1,\ldots,x_n$ with
 constant $1+\epsilon_0$ and functionals  $x_i^*$ such that
${\rm range}\ x_i^* \subseteq {\rm range}\ x_i$, $x_i^*(x_i) \geq 1/2$ for all
$i$, and such that $x_i^* \in Y_{2j}$ whenever $i \in E_j$.
Since
 $\norm{\sum_{i \in E_j}x_i} \leq  2nf(n)^{-1} m^{-1}$, it follows that
$$\norm{\sum_{i \in E_j}x_i^*} \geq \frac{n}{3m}\frac{mf(n)}{2n}=f(n)/6.$$

Let $z_j^*=\norm{\sum_{i \in E_j}x_i^*}^{-1}\sum_{i \in E_j}x_i^*$. Then $z_j^* \in
 Y_{2j}$ for all $j$ and
$$\norm{\sum_{j=1}^m z_j^*} \leq \frac{6}{f(n)}f(n)=6.$$
Therefore
$$\norm{\sum_{i=1}^m z_i^*} \leq 24f(m)^{-1}\norm{\sum_{i=1}^m y_i^*}
\leq \epsilon \norm{\sum_{i=1}^m y_i^*}.$$

\end{proof}

\begin{cor} The spaces $S$, $S^*$, $GM$,  $GM^*$, $G$,  $G^*$, $G_u$, and $G_u^*$ are uniformly
inhomogeneous. \end{cor}

\

A slightly stronger statement may be obtained by the
proof of Proposition \ref{spaceswithtcaciuc}, in the sense that the vectors
$y_i$ in the definition of uniform inhomogeneity may be chosen to be
successive. More explicitely, the conclusion may be replaced by the statement that
$$
\forall M\geq 1\;  \exists n \in \N\; \forall Y_1,\ldots,Y_{2n} \subseteq X\;
\exists y_i \in\ku S_{Y_i}\;(y_i)_{i=1}^n\not\sim_M(y_i)_{i=n+1}^{2n}.
$$
where $y_1<\cdots<y_n$ and $y_{n+1}<\cdots<y_{2n}$, and as before $Y_1,\dots,Y_{2n}$ are infinite-dimensional subspaces of $X$.

This property is therefore a block version of the property of uniform
inhomogeneity.  It was observed in \cite{FR4} that the
sixth dichotomy had the following ``block'' version: any Schauder basis of a Banach space
contains a block sequence which is either block uniformly inhomogeneous in the
above sense or asymptotically $\ell_p$ for some $p \in [1,+\infty]$.

It is interesting to observe that either side of this dichotomy corresponds to
one of the two main families of HI spaces, namely spaces
of the type of Gowers-Maurey, based on the example of Schlumprecht, and spaces of the
type of Argyros-Deliyanni, based on Tsirelson's type spaces.
More precisely, spaces of the type of Gowers-Maurey are block uniformly
inhomogeneous, while spaces of the type of Argyros-Deliyanni are
asymptotically $\ell_1$. Observe that the original dichotomy of Tcaciuc fails
to distinguish between these two families, since any HI space is trivially
uniformly inhomogeneous, see \cite{FR4}.

\section{Tight HI spaces of the type of Gowers and Maurey}\label{gowers-maurey}

In this section we show that Gowers' space $G$ constructed in \cite{g:asymptotic} and its dual
are of type
  (1).
  The proof is a refinement of the proof  that
  $G_u$ or $G_u^*$ is of type (3), in which we observe that the hypothesis of
  unconditionality may be replaced by asymptotic unconditionality. The idea is
  to produce constituent parts of vectors or functionals in Gowers'
  construction with sufficient control on their supports (and not just on
  their ranges, as  would be enough to obtain the HI property for example).

\subsection{ A HI space tight by range}
The space $G$ has a norm defined by induction as in $GM$, with the addition of
a new term which guarantees that its basis $(e_n)$ is $2$-asymptotically
unconditional, that is for any sequence of normalised vectors
$N<x_1<\ldots<x_N$, any sequence of scalars $a_1,\ldots,a_N$ and any sequence
of signs $\epsilon_1,\ldots,\epsilon_N$,
$$\norm{\sum_{n=1}^N \epsilon_n a_n x_n} \leq 2\norm{\sum_{n=1}^N a_n x_n}.$$
The basis is bimonotone and, although this is not stated in
\cite{g:asymptotic},  it may be proved as for $GM$ that $G$ is reflexive. It
follows that the dual basis of $(e_n)$ is also $2$-asymptotically
unconditional. The norm on $G$ is defined by the implicit
equation, for all $x \in c_{00}$:

$$\norm{x}=\norm{x}_{c_0}\vee\ \sup\Bigl\{f(n)^{-1}\sum_{i=1}^n\norm{E_i x}\Del 2 \leq n, E_1<\ldots<E_n\Bigr\}$$
$$\vee\ \sup\Bigl\{|x^*(Ex)|\Del k \in K, x^* \hbox{ special of length } k, E \subseteq \N\Bigr\}$$
$$\vee\ \sup\Bigl\{\norm{Sx}\Del S \hbox{ is an admissible operator}\Bigr\},$$

\

where $E$, $E_1,\ldots,E_n$ are intervals of integers, and $S$ is an {\em
admissible operator} if $Sx=\frac{1}{2}\sum_{n=1}^N \epsilon_n E_n x$ for some
sequence of signs $\epsilon_1,\ldots,\epsilon_N$ and some sequence
$E_1,\ldots,E_N$ of intervals which is {\em admissible}, i.e. $N<E_1$ and
$1+\max E_i=\min E_{i+1}$ for every $i < N$.

{\em R.I.S. pairs}  and {\em special pairs} are considered in \cite{g:asymptotic}; first we shall need
a more general definition of these.
Let $x_1,\ldots,x_m$ be a R.I.S. with constant $C$, $m \in [\log N, \exp N]$,
$N \in L$, and let
$x_1^*,\ldots, x_m^*$ be successive normalised functionals. Then we
call {\em  generalised R.I.S. pair with constant $C$} the pair
$(x,x^*)$ defined by
$x=\norm{\sum_{i=1}^m x_i}^{-1}(\sum_{i=1}^m x_i)$ and
$x^*=f(m)^{-1}\sum_{i=1}^m x_i^*$.

Let  $z_1,\ldots,z_k$ be a sequence of successive normalised R.I.S. vectors with constant
$C$, and let
$z_1^*,\ldots, z_k^*$ be a special sequence such that $(z_i,z_i^*)$ is a generalized
R.I.S. pair for each $i$. Then we shall call {\em generalised special pair
with constant $C$} the pair
$(z,z^*)$ defined by
$z=\sum_{i=1}^k z_i$ and
$z^*=f(k)^{-1/2}(\sum_{i=1}^k z_i^*)$.
The pair $(\norm{z}^{-1}z,z^*)$ will be called {\em normalised generalised
  special pair}.

\begin{lemme}\label{critical}
 Let $(z,z^*)$ be a generalised special pair in $G$, of length $k \in K$, with
 constant $2$ and such that
${\rm supp}\ z^* \cap {\rm supp}\ z = \emptyset$.
Then
$$\norm{z} \leq \frac{5k}{f(k)}.$$
\end{lemme}

\begin{proof}
The proof follows classically the methods of \cite{GM} or
\cite{g:hyperplanes}.
 Let $K_0=K \setminus \{k\}$ and let $g$ be the corresponding function given
by \cite{g:asymptotic} Lemma 5. To prove that $\norm{z} \leq5kf(k)^{-1}$ it is
enough by Lemma \ref{fundamental} to prove that
for any interval $E$ such that $\norm{Ez} \geq 1/3$, $Ez$ is normed by some $(M,g)$-form with $M \geq 2$.

By the discussion in \cite{g:asymptotic} after the definition of the norm, the
only possible norming functionals apart from $(M,g)$-forms are of the form
$Sw^*$ where $w^*$ is a special functional of length $k$, and $S$ is an
``acceptable'' operator according to the terminology of \cite{g:asymptotic}. We shall not  state the definition of an acceptable
operator $S$, we shall just need to know  that since such an operator is
diagonal of norm at most $1$, it preserves support and $(M,g)$-forms,
\cite{g:asymptotic} Lemma 6. So let $w^*=f(k)^{-1/2}(w_1^*+\cdots+w_k^*)$ be  a
special functional of length $k$, $S$ be an acceptable operator, and $E$ be an
interval such that $\norm{Ez} \geq 1/3$. We need to show that $Sw^*$ does not
norm $Ez$.

Let $t$ be minimal such that $w_t^* \neq z_t^*$. If $i \neq j$ or $i=j>t$ then
by definition of special sequences there exist $M \neq N \in L$, $\min(M,N)
\geq j_{2k}$, such that $w_i^*$ and therefore $Sw_i^*$ is an $(M,f)$-form and
$z_j$ is an R.I.S. vector of size $N$ and constant $2$. By \cite{g:asymptotic}
Lemma 8, $z_j$ is an $\ell_{1+}^{N^{1/10}}$-average with constant $4$. If $M<N$
then $2M<\log \log \log N$ so, by \cite{g:asymptotic} Lemma 2, $|Sw_i^*(Ez_j)|
\leq 12f(M)^{-1}$. If $M>N$ then $\log \log \log M>2N$  so, by
\cite{g:asymptotic} Lemma 3, $|Sw_i^*(Ez_j)| \leq 3f(N)/N$. In both cases it
follows that $|Sw_i^*(Ez_j)| \leq k^{-2}$.

If $i=j=t$ we simply have $|Sw_i^*(Ez_j)| \leq 1$. Finally if $i=j<t$ then $w_i^*=z_i^*$. and
 since
${\rm supp}\ Sz^*_i \subseteq {\rm supp}\ z_i^*$
and
${\rm supp}\ Ez_i \subseteq {\rm supp}\ z_i,$
it follows that $Sw_i^*(Ez_j)=0$ in this case.

Summing up we have obtained that
$$|Sw^*(Ez)| \leq f(k)^{-1/2}(k^2. k^{-2}+1)=2f(k)^{-1/2} < 1/3 \leq \norm{Ez}.$$
Therefore $Sw^*$ does not norm $Ez$ and this finishes the proof.
\end{proof}

The next lemma is expressed in a version which may seem technical but this will
make the proof that $G$ is of type (1) more pleasant to read. At  first
reading, the reader may simply assume that $T=Id$ in its hypothesis.

\begin{lemme}\label{average}
Let $n \in \N$ and let $\epsilon>0$. Let $(x_i)_i$ be a normalised block basis
in $G$
of length  $n^k$ and supported after $2n^k$, where $k=\min\{i \del f(n^i) < (1+\epsilon)^i\}$, and $T:[x_i] \rightarrow G$ be
an isomorphism such that $(Tx_i)$ is also a normalised block basis. Then for
any $n \in \N$ and $\epsilon>0$, there exist a finite interval $F$ and a
multiple $x$ of $\sum_{i \in F}x_i$ such that $Tx$ is an $\ell_{1+}^n$-average
with constant $1+\epsilon$, and a normalised functional $x^*$ such that $x^*(x)
>1/2$ and ${\rm supp}\ x^* \subseteq \bigcup_{i \in F}{\rm range}\ x_i$.
\end{lemme}

\begin{proof}
The proof from \cite{g:asymptotic} that the block basis $(Tx_i)$  contains an
$\ell_{1+}^n$-average with constant $1+\epsilon$ is the same as for $GM$, and
gives that such a vector exists of the form $Tx=\lambda \sum_{i \in F}Tx_i$,
thanks to the condition on the length of $(x_i)$.  We may
therefore deduce that $2|F|-1<{\rm supp}\ x$. Let $y^*$ be a unit functional
which norms $x$ and such that ${\rm range}\; y^* \subseteq {\rm range}\; x$.
Let $x^*=Ey^* $ where $E$ is  the union of the $|F|$ intervals ${\rm range}\;
x_i, i \in F$. Then $x^*(x)=y^*(x)=1$ and by unconditional asymptoticity of
$G^*$,  $\norm{x^*} \leq \frac{3}{2}\norm{y^*}<2$.
 \end{proof}

The proof that $G$ is HI requires defining ``extra-special sequences'' after
having defined special sequences in the usual $GM$ way. However, to prove that
$G$ is tight by range, we shall not need to enter that level of complexity and
shall just use  special sequences.

\begin{prop}\label{type1} The space $G$ is of type (1). \end{prop}

\begin{proof}
Assume some normalised block-sequence $(x_n)$ is such that $[x_n]$ embeds into
$Y=[e_i, i \notin \bigcup_n {\rm range}\ x_n]$ and look for a contradiction.
Passing to a subsequence and by reflexivity we may  assume that there is some
isomorphism $T:[x_n] \rightarrow Y$ satisfying the hypothesis of Lemma
\ref{average}, that is, $(Tx_n)$ is a normalised block basis in $Y$. Fixing
$\epsilon=1/10$ we may construct by Lemma \ref{average} some block-sequence of
vectors in $[x_n]$ which are $1/2$-normed by functionals in ${\bf Q}$ of
support included in $\bigcup_n {\rm range}\; x_n$, and whose images by $T$ form
a sequence of increasing length ${\ell}_{1+}^n$-averages with constant
$1+\epsilon$. If $Tz_1,\ldots,Tz_m$ is a R.I.S. of these
${\ell}_{1+}^n$-averages with constant $1+\epsilon$, with $m \in [\log N, \exp
N]$, $N \in L$, and $z_1^*,\ldots,z_m^*$ are the functionals associated to
$z_1,\ldots,z_m$, then by \cite{g:asymptotic} Lemma 7, the $(m,f)$-form
$z^*=f(m)^{-1}(z_1^*+\ldots+z_m^*)$ satisfies
$$
z^*(z_1+\ldots+z_m) \geq \frac{m}{2f(m)} \geq \frac{1}{4}\norm{Tz_1+\ldots+Tz_m} \geq (4\norm{T^{-1}})^{-1}\norm{z_1+\dots+z_m}.
$$
Therefore we may build R.I.S.  vectors $Tz$ with constant $1+\epsilon$ of
arbitrary length $m$  in $[\log N, \exp N]$,  $N \in L$, so that $z$ is
$(4\norm{T^{-1}})^{-1}$-normed by an $(m,f)$-form $z^*$ of support included in
$\bigcup_n {\rm range}\;x_n$. For such $(z,z^*)$,  $(Tz,z^*)$ is a generalised
R.I.S. pair. We then consider a  sequence $Tz_1,\ldots,Tz_k$ of length $k \in
K$ of such R.I.S. vectors, such that there exists some special sequence of
corresponding functionals $z_1^*,\ldots,z_k^*$, and finally the  pair $(z,z^*)$
where $z=z_1+\cdots+z_k$ and $z^*=f(k)^{-1/2}(z_1^*+\ldots+z_k^*)$: observe
that  the support of $z^*$ is still included in $\bigcup_n {\rm range}\;x_n$.
 Since $(Tz,z^*)$ is a generalised special pair, it follows from
Lemma \ref{critical} that $$\norm{Tz} \leq 5kf(k)^{-1}.$$ On the other hand,
$$\norm{z} \geq z^*(z) \geq (4\norm{T^{-1}})^{-1}k f(k)^{-1/2}.$$
 Since $k$ was arbitrary in $K$ this implies that $T^{-1}$ is unbounded and provides the desired contradiction.
\end{proof}


\subsection{A HI space tight by range and locally minimal}
As we shall now prove, the dual $G^*$ of $G$ is of type (1) as well, but also
locally minimal.

\begin{lemme}\label{linftynbis}
Let $(x_i^*)$ be a normalised block basis in $G^*$. Then for any $n \in \N$
and $\epsilon>0$, there exists $N(n,\epsilon)$, a finite interval $F
\subseteq [1,N(n,\epsilon)]$, a multiple $x^*$ of
$\sum_{i \in F}x_i^*$ which is an $\ell_{\infty +}^n$-average with constant
$1+\epsilon$ and an $\ell_{1+}^n$-average $x$ with constant $2$ such that
$x^*(x) >1/2$ and ${\rm supp}\ x \subseteq \bigcup_{i \in F}{\rm range}\ x_i^*$.
\end{lemme}

\begin{proof}
We may assume that $\epsilon<1/6$. By Lemma \ref{linftyn} we may find for each
$i
\leq n$
an interval $F_i$, with $|F_i| \leq 2\min F_i$, and a vector $y_i^*$ of
the form $\lambda \sum_{k \in F_i} x_k^*$, such that
$y^*=\sum_{i=1}^n y_i^*$ is an $\ell_{\infty +}^n$-average with constant
$1+\epsilon$. Let, for each $i$, $x_i$ be normalised such that
$y_i^*(x_i)=\norm{y_i^*}$ and ${\rm range}\ x_i \subseteq {\rm range}\ y_i^*$.
Let $y_i=E_i x_i$, where $E_i$ denotes the canonical projection
on $[e_m, m \in \bigcup_{k \in F_i}{\rm range}\ x_k^*]$. By the asymptotic
unconditionality of $(e_n)$, we have that $\norm{y_i} \leq 3/2$.
Let $y_i^{\prime}=\norm{y_i}^{-1}y_i$, then
$$y_i^*(y_i^{\prime})=\norm{y_i}^{-1}y_i^*(y_i)=\norm{y_i}^{-1}y_i^*(x_i) \geq
\frac{2}{3}\norm{y_i^*}.$$
By Lemma \ref{linftyn}, the vector $x=\sum_i y_i^{\prime}$ is an
$\ell_{1+}^n$-vector with constant  $2$, such that $x^*(x)
>\norm{x}/2$, and clearly  ${\rm supp}\ x \subseteq \bigcup_{i \in F}{\rm range}\ x_i^*$.
\end{proof}

\begin{prop} The space $G^*$ is locally minimal and tight by range.
\end{prop}

\begin{proof}

By Lemma \ref{linftynbis}  we may find in any finite block subspace of $G^*$ of length
$N(n,\epsilon)$ and supported after $N(n,\epsilon)$ an
$\ell_{\infty+}^n$-average with constant $1+\epsilon$. By asymptotic
unconditionality
we deduce that uniformly, any block-subspace of $G^*$ contains
$\ell_{\infty}^n$'s, and therefore $G^*$ is locally minimal.

We prove that $G^*$ is tight by range. Assume towards a contradiction that
some normalised block-sequence $(x_n^*)$ is such that $[x_n^*]$ embeds into
$Y=[e_i^*, i \notin \bigcup_n {\rm range}\ x_n^*]$ and look for a
contradiction. If $T$ is the associated isomorphism, we may by passing to a
subsequence and perturbating $T$ assume that $Tx_n^*$ is successive.

Let $\epsilon=1/10$. By Lemma \ref{linftynbis}, we find in $[x_k^*]$ and for
each $n$, an $\ell_{\infty+}^n$-average
$y_n^*$ with constant $1+\epsilon$ and
an $\ell_{1+}^n$-average $y_n$ with constant $2$, such that
$y_n^*(y_n) > 1/2$ and ${\rm supp}\ y_n \subseteq \bigcup_k {\rm range}\ x_k^*$.
By construction, for each $n$, $Ty_n^*$ is disjointly supported from $[x_k^*]$,
and up to modifying $T$, we may assume that $Ty_n^*$  is in ${\bf Q}$ and of
norm at most $1$ for each $n$.

 If $z_1,\ldots,z_m$ is a R.I.S. of these ${\ell}_{1+}^n$-averages $y_n$ with constant $2$,
with $m \in [\log N, \exp N]$, $N \in L$, and $z_1^*,\ldots,z_m^*$ are the
${\ell}_{\infty+}^n$-averages  associated to $z_1,\ldots,z_m$, then by \cite{g:hyperplanes} Lemma 7, the $(m,f)$-form
$z^*=f(m)^{-1}(z_1^*+\ldots+z_m^*)$ satisfies
$$z^*(z_1+\ldots+z_m) \geq \frac{m}{2f(m)} \geq
\frac{1}{6}\norm{z_1+\ldots+z_m},$$
and furthermore  $Tz^*$ is also an $(m,f)$-form.
Therefore we may build R.I.S.  vectors $z$ with constant $2$ of
arbitrary length $m$  in $[\log N, \exp N]$,  $N \in L$, so that $z$ is
$6^{-1}$-normed by an $(m,f)$-form $z^*$ such that $Tz^*$ is also
an $(m,f)$-form.
We may then consider a  sequence $z_1,\ldots,z_k$ of length $k \in K$ of such
R.I.S. vectors of length $m_i$, and some corresponding functionals $z_1^*,\ldots,z_k^*$ (i.e.,
$z_i^*$ $6^{-1}$-norms $z_i$ and $Tz_i^*$ is also an $(m_i,f)$-form for all $i$), such
that
$Tz_1^*,\ldots,Tz_k^*$ is a special sequence.
Then we let
$z=z_1+\cdots+z_k$ and $z^*=f(k)^{-1/2}(z_1^*+\ldots+z_k^*)$, and observe that
$(z,Tz^*)$ is a generalised special pair.
Since $Tz^*=f(k)^{-1/2}(Tz_1^*+\ldots+Tz_k^*)$ is a special functional it follows that
$$\norm{Tz^*} \leq 1.$$ But it follows from Lemma \ref{critical} that  $\norm{z} \leq 5kf(k)^{-1}$.
Therefore
$$\norm{z^*} \geq z^*(z)/\norm{z} \geq  f(k)^{1/2}/30.$$
 Since $k$ was arbitrary in $K$ this implies that $T^{-1}$ is unbounded and provides the desired contradiction.
\end{proof}

It remains to check that $G^*$ is HI. The proof is very similar to the one in
 \cite{g:asymptotic} that $G$ is HI, and
 we shall therefore not give all details. There are two main differences
 between the two proofs. In \cite{g:asymptotic} some special vectors and
 functionals are constructed,  the vectors are taken alternatively in
 arbitrary block subspaces $Y$ and $Z$ of $G$, and no condition is imposed on where to pick the functionals. In our case there is no condition on where to choose the vectors but  we need to pick the
 functionals in arbitrary subspaces $Y$ and $Z$ of $G^*$ instead. This  is
 possible because of Lemma \ref{linftynbis}. We also need to correct what seems to be a slight imprecision in the proof of
 \cite{g:asymptotic} about the value of some normalising factors,   and therefore we also get worst constants for our
 estimates.

Let $\epsilon=1/10$. Following Gowers we define  an {\em R.I.S. pair} of size $N$ to be a generalised R.I.S.
pair $(x,x^*)$ with constant $1+\epsilon$ of the form $(\norm{x_1+\ldots+x_N}^{-1}(x_1+\ldots+x_N),
f(N)^{-1}(x_1^*+\dots+x_N^*))$, where $x_n^*(x_n)\geq 1/3$ and
${\rm range}\ x_n^* \subset{\rm range}\ x_n$ for each $n$.
A {\em special pair} is a normalised generalised special pair with constant $1+\epsilon$ of the form
  $(x,x^*)$ where $x=\norm{x_1+\ldots+x_k}^{-1}(x_1+\ldots+x_k)$  and
  $x^*=f(k)^{-1/2}(x_1^*+\dots+x_k^*)$ with ${\rm range}\ x_n^* \subseteq {\rm range}\ x_n$
and  for each $n$, $x_n^*\in\bf
Q$, $|x_n^*(x_n)-1/2|<10^{-\min{\rm supp}\ x_n}$.
By \cite{g:asymptotic} Lemma 8, $z$ is a R.I.S. vector with constant $2$
whenever $(z,z^*)$ is a special pair. We shall also require that
 $k \leq \min{\rm supp}\ x_1$, which will imply by \cite{g:asymptotic} Lemma 9 that for
$m<k^{1/10}$, $z$ is a $\ell_{1+}^m$-average with constant 8 (see the
beginning of the proof of Proposition \ref{GstarHI}).

Going  up a level of ``specialness'', a {\em special R.I.S.-pair}
is a generalised  R.I.S.-pair with constant 8 of the form
$(\norm{x_1+\ldots+x_N}^{-1}(x_1+\ldots+x_N),
f(N)^{-1}(x_1^*+\dots+x_N^*))$, where
${\rm range}\ x_n^* \subset{\rm range}\ x_n$ for each $n$, and with
 the additional condition that $(x_n,x_n^*)$ is a special
pair of length at least $\min {\rm supp}\ x_n$.  Finally, an {\em extra-special pair} of
size $k$ is a normalised generalised special pair $(x,x^*)$ with constant 8 of the form
  $x=\norm{x_1+\ldots+x_k}^{-1}(x_1+\ldots+x_k)$  and
  $x^*=f(k)^{-1/2}(x_1^*+\dots+x_k^*)$ with ${\rm range}\ x_n^* \subseteq {\rm
    range}\ x_n$, such that, for each $n$,
$(x_n,x_n^*)$ is a special R.I.S.-pair of length
$\sigma(x_1^*,\dots,x_{n-1}^*)$.

\

Given $Y,Z$ block subspaces of $G^*$ we shall show how to find an
extra-special pair $(x,x^*)$ of size $k$, with $x^*$ built out of vectors in
$Y$ or $Z$, such that the signs of these constituent parts of $x^*$ can be
changed according to belonging to $Y$ or $Z$ to produce a vector
$x^{\prime*}$ with $\norm{x^{\prime*}}\leq 12f(k)^{-1/2}\norm{x^*}$. This will
then prove the result.

Consider then an extra-special pair $(x,x^*)$. Then $x$ splits up as
$$\nu^{-1}\sum_{i=1}^k\nu_i^{-1}\sum_{j=1}^{N_i}\nu_{ij}^{-1}
\sum_{r=1}^{k_{ij}}x_{ijr}$$
 and $x^*$ as
$$f(k)^{-1/2}\sum_{i=1}^k f(N_i)^{-1}\sum_{j=1}^{N_i}f(k_{ij})^{-1}
\sum_{r=1}^{k_{ij}}x^*_{ijr}\,$$
where
the numbers $\nu$, $\nu_i$ and $\nu_{ij}$ are the norms of what
appears to the right. These special sequences are
chosen far enough ``to the right'' so that $k_{ij}\leq\min{\rm supp}\ x_{ij1}$,
and also so that $(\max{\rm supp}\ x_{i\,j-1})^2k_{ij}^{-1}\leq 4^{-(i+j)}$.
We shall also write $x_i$ for $\nu_i^{-1}\sum_{j=1}^{N_i}\nu_{ij}^{-1}
\sum_{r=1}^{k_{ij}}x_{ijr}$ and $x_{ij}$ for $\nu_{ij}^{-1}
\sum_{r=1}^{k_{ij}}x_{ijr}$.

We define a  vector $x'$ by
$$\sum_{i=1}^k\nu_i^{\prime -1}\sum_{j=1}^{N_i}\nu_{ij}^{\prime -1}
\sum_{r=1}^{k_{ij}}(-1)^rx_{ijr},$$
where
the numbers  $\nu_i^{\prime}$ and $\nu_{ij}^{\prime}$ are the norms of what
appears to the right.
We shall write $x'_i$ for $\nu_i^{\prime -1}\sum_{j=1}^{N_i}\nu_{ij}^{\prime -1}
\sum_{r=1}^{k_{ij}}(-1)^rx_{ijr}$ and $x'_{ij}$ for $\nu_{ij}^{\prime -1}
\sum_{r=1}^{k_{ij}}(-1)^rx_{ijr}$.

Finally we define a functional $x^{\prime *}$ as
$$f(k)^{-1/2}\sum_{i=1}^k f(N_i)^{-1}\sum_{j=1}^{N_i}f(k_{ij})^{-1}
\sum_{r=1}^{k_{ij}}(-1)^k x^*_{ijr}.$$

\begin{prop}\label{GstarHI} The space $G^*$ is HI.\end{prop}

\begin{proof}
Fix $Y$ and $Z$ block subspaces of $G^*$.
By Lemma \ref{linftynbis} we may construct an extra-special pair $(x,x^*)$ so  that $x^*_{ijr}$ belongs to $Y$ when
$r$ is odd and to $Z$ when $r$ is even.

We first discuss the normalisation of the vectors involved in the definition of $x'$.
By the increasing condition on $k_{ij}$ and $x_{ijr}$ and by asymptotic
unconditionality, we have that
$$\norm{\sum_{r=1}^{k_{ij}} (-1)^r x_{ijr}} \leq 2
\norm{\sum_{r=1}^{k_{ij}}  x_{ijr}},$$
which means that $\nu^{\prime}_{ij} \leq 2 \nu_{ij}$.
Furthermore it also follows  that
 the functional $(1/2)f(k_{ij})^{-1/2}\sum_{r=1}^{k_{ij}}(-1)^rx^*_{ijr}$ is
 of norm at most $1$, and therefore we have that
 $\norm{\sum_{r=1}^{k_{ij}}(-1)^rx_{ijr}}\geq
(1/2)k_{ij}f(k_{ij})^{-1/2}$.
 Lemma 9 from \cite{g:asymptotic} therefore tells us that, for
every $i,j$, $x'_{ij}$ is an $\ell_{1+}^{m_{ij}}$-average with
constant 8, if $m_{ij}<k_{ij}^{1/10}$.
But the $k_{ij}$ increase so fast that, for any $i$, this implies that
the sequence $x'_{i1},\dots,x'_{i\,N_i}$ is a rapidly increasing
sequence with constant 8.
By \cite{g:asymptotic} Lemma 7, it follows that
$$\norm{\sum_{j=1}^{N_i} x_{ij}^{\prime}} \leq 9 N_i/f(N_i).$$
Therefore by the $f$-lower estimate in $G$
we have that $\nu^{\prime}_i \leq 9\nu_i$.

We shall now prove that $\norm{x'} \leq 12kf(k)^{-1}$.
This will imply that
$$\norm{x_*^{\prime}} \geq \frac{x_*^{\prime}(x')}{\norm{x'}}
\geq \frac{f(k)}{12k}[f(k)^{-1/2}
\sum_{i=1}^k f(N_i)^{-1}\nu_i^{\prime -1}\sum_{j=1}^{N_i}f(k_{ij})^{-1}\nu_{ij}^{\prime -1}
\sum_{r=1}^{k_{ij}}x_{ijr}^*(x_{ijr})]$$
$$\geq f(k)^{1/2}(12k)^{-1}.18^{-1}
[\sum_{i=1}^k f(N_i)^{-1}\nu_i^{-1}\sum_{j=1}^{N_i}f(k_{ij})^{-1}\nu_{ij}^{-1}
\sum_{r=1}^{k_{ij}}x_{ijr}^*(x_{ijr})]$$
$$= f(k)^{1/2} (216k)^{-1} \sum_{i=1}^k x_i^*(x_i) \geq 648^{-1}
f(k)^{1/2}.$$
By construction of $x^*$ and $x^{\prime *}$ this will imply that
$$\norm{y^*-z^*} \geq 648^{-1}f(k)^{1/2}\norm{y^*+z^*}$$
for some non zero $y^* \in Y$ and $z^* \in Z$, and since $k \in K$ was arbitrary,
as well as $Y$ and $Z$, this will prove that $G^*$ is HI.

\

The proof that $\norm{x'} \leq 12kf(k)^{-1}$ is given in three steps:

\

\paragraph{Step 1} {\em The vector $x'$ is a R.I.S. vector with constant 11.}
\begin{proof} We already know the sequence $x'_{i1},\dots,x'_{i\,N_i}$ is a rapidly increasing
sequence with constant 8.  Then by \cite{g:asymptotic} Lemma 8 we get that $x'_i$ is also
an $\ell_{1+}^{M_i}$-average with constant 11, if $M_i<N_i^{1/10}$.
Finally, this implies that $x'$ is an R.I.S.-vector with constant 11,
as claimed.\end{proof}

\paragraph{Step 2} {\em Let $K_0=K\setminus\{k\}$, let $g\in\mathcal F$ be
the corresponding function given by \cite{g:asymptotic} Lemma 5.
For every interval $E$ such that $\norm{Ex'}\geq 1/3$, $Ex'$ is normed by
an $(M,g)$-form.}

\begin{proof}The proof is exactly the same as the one of Step 2 in the proof
  of Gowers concerning $G$,
apart from some constants which are modified due to the change of  constant in
Step 1 and to the normalising constants relating $\nu_i$ and $\nu_{ij}$
respectively to
$\nu_i^{\prime}$ and $\nu_{ij}^{\prime}$. The reader is therefore referred to
\cite{g:asymptotic}.\end{proof}

\paragraph{Step 3} {\em The norm of $x'$ is at most $12kg(k)^{-1}=12kf(k)^{-1}$}
\begin{proof} This is an immediate consequence of Step 1, Step 2 and of
 Lemma \ref{fundamental}.
\end{proof}

We conclude that the space $G^*$ is HI, and thus locally minimal of type (1). \end{proof}







\section{Unconditional tight spaces of the type of Argyros and Deliyanni}\label{argyros}
By Proposition \ref{spaceswithtcaciuc}, unconditional or HI spaces built on the model of
Gowers-Maurey's spaces are uniformly inhomogeneous (and even block uniformly inhomogeneous). We shall now
consider a space of Argyros-Deliyanni type, more specifically of the type of a
space constructed by Argyros, Deliyanni, Kutzarova and Manoussakis
\cite{ADKM}, with the
opposite property, i.e., with a basis which is strongly
asymptotically $\ell_1$. This space will also be tight by support and therefore will not contain a copy of $\ell_1$. By the implication at the end of the diagram which appears just before Theorem \ref{final}, this basis will therefore be tight with
constants as well, making this example the ``worst'' known so far in
terms of minimality.

 Again in this section block vectors will not necessarily
be normalized  and some familiarity with the construction in
\cite{ADKM} will be assumed.

\subsection{A strongly asymptotically $\ell_1$ space tight by support}

In \cite{ADKM} an example of HI space $X_{hi}$ is constructed, based
on a ``boundedly modified'' mixed Tsirelson space $X_{M(1),u}$. We
shall construct an unconditional version $X_u$ of $X_{hi}$ in a
similar way as $G_u$ is an unconditional version of $GM$. The proof
that $X_u$ is of type (3) will be based on the proof that $X_{hi}$ is HI, conditional
estimates in the proof of \cite{ADKM}  becoming essentially trivial
in our case due to disjointness of supports.

Fix  a basis $(e_n)$ and ${\mathcal M}$ a family of finite subsets
of $\N$. Recall that a family $x_1,\ldots,x_n$ is {\em ${\mathcal
M}$-admissible} if $x_1<\cdots<x_n$ and $\{\min {\rm supp}\
x_1,\ldots,\min {\rm supp}\ x_n\} \in {\mathcal M}$, and {\em
${\mathcal M}$-allowable} if $x_1,\ldots,x_n$ are vectors with
disjoint supports such that $\{\min {\rm supp}\ x_1,\ldots,\min {\rm
supp}\ x_n\} \in {\mathcal M}$. Let ${\mathcal S}$ denote the family
of Schreier sets, i.e., of subsets $F$ of $\N$ such that $|F| \leq \min F$,
${\mathcal M}_j$ be the subsequence of the sequence $({\mathcal
F}_k)$ of Schreier families associated to sequences of integers
$t_j$ and $k_j$ defined in \cite{ADKM} p 70.

We need to define a new notion. For $W$ a set of functionals which
is stable under projections onto subsets of $\N$, we let ${\rm
conv}_{\Q}W$ denote the set of rational convex combinations of
elements of $W$. By the stability property of $W$ we may write any
$c^* \in {\rm conv}_{\Q}W$ as  a rational convex combination of the
form $\sum_i \lambda_i x_i^*$ where $x_i^* \in W$ and ${\rm supp}\
x_i^* \subseteq {\rm supp}\ c^*$ for each $i$. In this case the set
$\{x_i^* \}_i$ will be called a $W$-compatible decomposition of
$c^*$, and we let $W(c^*) \subseteq W$ be the union of all
$W$-compatible decompositions of $c^*$. Note that if ${\mathcal M}$
is a family of finite subsets of $\N$, $(c_1^*,\ldots,c_d^*)$ is
${\mathcal M}$-admissible, and $x_i^* \in W(c_i^*)$ for all $i$,
then $(x_1^*,\ldots,x_d^*)$ is also ${\mathcal M}$-admissible.

Let ${\mathcal B}=\{\sum_{n}\lambda_n e_n: (\lambda_n)_n \in c_{00},
\lambda_n \in \Q \cap [-1,1]\}$ and let $\Phi$ be a 1-1 function
from ${\mathcal B}^{<\N}$ into $2\N$ such that if
$(c_1^*,\ldots,c_k^*) \in {\mathcal B}^{<\N}$, $j_1$ is minimal such
that $c_1^* \in {\rm conv}_{\Q}{\mathcal A}_{j_1}$, and
$j_l=\Phi(c_1^*,\ldots,c_{l-1}^*)$ for each $l=2,3,\ldots$, then
$\Phi(c_1^*,\ldots,c_k^*)>\max\{j_1,\ldots,j_k\}$ (the set
${\mathcal A}_j$ is defined in \cite{ADKM} p 71 by ${\mathcal
A}_j=\cup_n(K_j^n \setminus K^0)$ where the $K_j^n$'s are the sets
corresponding to the inductive definition of $X_{M(1),u}$).

For $j=1,2,\ldots $, we set $L_j^0=\{ \pm e_n:n\in \N\}$. Suppose
that $\{ L_j^n\}_{j=1}^{\infty }$ have been defined. We set
$L^n=\cup_{j=1}^{\infty}L^n_j$ and
$$L_1^{n+1}=\pm L_1^n\cup\{\frac{1}{2}(x_1^{*}+\ldots +x_d^{*}):d\in \N,x_i^{*}\in L^n,$$

$$(x_1^{*},\ldots,x_d^{*})\;{\rm is}\; \ {\mathcal S}-{\rm allowable}\},$$
\noindent and for $j\geq 1$,
$$L_{2j}^{n+1}=\pm L_{2j}^n\cup\{\frac{1}{m_{2j}}(x_1^{\ast }+\ldots
+x_d^{*}):d\in {\N},x_i^{*}\in L^n,$$
$$(x_1^{*},\ldots ,x_d^{*})\;{\rm is}\;
{\mathcal M}_{2j}-{\rm admissible }\},$$

$$L_{2j+1}^{\prime\;n+1}=\pm L_{2j+1}^n\cup \{\frac{1}{m_{2j+1}}(x_1^{\ast
}+\ldots +x_d^{\ast }):d\in {\N} {\rm\ such\ that}$$
$$\exists (c_1^*,\ldots,c_d^*)\;{\mathcal M}_{2j+1}-{\rm admissible\ and\ }
k>2j+1 {\rm\ with\ }c_1^* \in {\rm conv}_{\Q}L_{2k}^n, x_1^*\in
L_{2k}^n(c_1^*),$$

$$c_i^* \in {\rm conv}_{\Q}L_{\Phi (c_1^{\ast },\ldots ,c_{i-1}^{\ast })}^n,
x_i^{\ast }\in L_{\Phi (c_1^{\ast },\ldots ,c_{i-1}^{\ast
})}^n(c_i^*) \;{\rm for}\;1<i\leq d\},$$

$$L_{2j+1}^{n+1}=\{ Ex^{\ast }:x^{\ast }\in L_{2j+1}^{\prime\;n+1},
 E {\rm\ subset\ of\ } \N\}.$$

We set ${\mathcal B}_j=\cup_{n=1}^{\infty }(L_j^n\setminus L^0)$
 and we consider the norm on $c_{00}$
defined by the set $L=L^0\cup (\cup_{j=1}^{\infty }{\mathcal B}_j)$.
The space $X_u$ is the completion of $c_{00}$ under this norm.

\

In \cite{ADKM} the space $X_{hi}$ is defined in the same way except
that $E$ is an {\bf interval} of integers in the definition of
$L_{2j+1}^{n+1}$, and the definition of $L_{2j+1}^{\prime\;n+1}$ is
simpler, i.e., the coding $\Phi$ is defined directly on ${\mathcal
M}_{2j+1}$-admissible families $x_1^*,\ldots,x_d^*$ in $L^{<\N}$ and
in the definition each $x_i^*$ belongs to $L_{\Phi
(x_1^{\ast},\ldots ,x_{i-1}^{\ast})}^n$. To prove the desired
properties for $X_u$ one could use the simpler definition of
$L_{2j+1}^{\prime\;n+1}$; however this definition doesn't seem to
provide enough special functionals to obtain interesting properties
for the dual as well.

The ground space for $X_{hi}$ and for $X_u$ is the space
$X_{M(1),u}$ associated to a norming set $K$ defined by the same
procedure as $L$, except that $K_{2j+1}^n$ is defined in the same
way as $K_{2j}^n$, i.e.
$$K_{2j}^{n+1}=\pm K_{2j}^n\cup\{\frac{1}{m_{2j}}(x_1^{\ast }+\ldots
+x_d^{*}):d\in {\N},x_i^{*}\in K^n,$$
$$(x_1^{*},\ldots ,x_d^{*})\;{\rm is}\;
{\mathcal M}_{2j+1}-{\rm admissible }\}.$$

For $n=0,1,2,\ldots ,$ we see  that $L_j^n$ is a subset of $K_j^n$,
and therefore $L \subseteq K$.
 The norming set $L$ is closed under
projections onto {\bf subsets} of $\N$, from which it follows that
its canonical basis is unconditional, and has the property that for
every $j$ and every ${\mathcal M}_{2j}$--admissible family $f_1,
f_2, \ldots f_d$ contained in $L$, $f=\frac{1}{m_{2j}}(f_1+\cdots
+f_d)$ belongs to $L$. The {\em weight} of such  an $f$ is defined
by $w(f)=1/m_{2j}$.  It follows that for every $j=1,2,\ldots$ and
every
 ${\mathcal M}_{2j}$--admissible family $x_1<x_2<\ldots<x_n$ in $X_u$,
$$\|\sum_{k=1}^nx_k\|\geq\frac{1}{m_{2j}}\sum_{k=1}^n\| x_k\|.$$ Likewise, for
${\mathcal S}$--allowable families $f_1,\ldots,f_n$ in $L$, we have
$f=\frac{1}{2}(f_1+\cdots+f_d) \in L$, and we define $w(f)=1/2$. The
weight is defined similarly in the case $2j+1$.

\begin{lemme} The canonical basis of $X_u$ is strongly asymptotically
  $\ell_1$.
\end{lemme}

\begin{proof} Fix $n\leq x_1,\ldots,x_n$ where $x_1,\ldots,x_n$ are normalised and
  disjointly supported. Fix $\epsilon>0$ and let for each $i$, $f_i \in L$ be
  such that $f_i(x_i) \geq (1+\epsilon)^{-1}$ and ${\rm supp}\ f_i \subseteq
{\rm supp}\ x_i$. The condition on the supports may be imposed
because $L$ is
  stable under projections onto subsets of $\N$.
Then $\frac{1}{2}\sum_{i=1}^n \pm f_i \in L$ and therefore
$$\norm{\sum_{i=1}^n \lambda_i x_i} \geq \frac{1}{2}\sum_{i=1}^n |\lambda_i|
f_i(x_i) \geq \frac{1}{2(1+\epsilon)}\sum_{i=1}^n |\lambda_i|,$$ for
any $\lambda_i$'s. Therefore $x_1,\ldots,x_n$ is $2$-equivalent to
the canonical basis of $\ell_1^n$.
\end{proof}

It remains to prove that $X_u$ has type (3). Recall that an analysis
$(K^s(f))_s$ of $f \in K$ is a decomposition of $f$ corresponding to
the inductive definition of $K$, see the precise definition in
Definition 2.3 \cite{ADKM}. We shall combine three types of
arguments. First $L$ was constructed so that $L \prec K$, which
means essentially that each $f \in L$ has an analysis $(K^s(f))_s$
whose elements actually belong to $L$ (see the definition on page 74
of \cite{ADKM}); so  all the results obtained in Section
2 of \cite{ADKM} for spaces defined through arbitrary $\tilde{K} \prec K$ (and in particular the crucial
Proposition 2.9) are valid in our case. Then we shall produce
estimates similar to those valid for $X_{hi}$ and  which are of two
forms: unconditional estimates, in which case the proofs from
\cite{ADKM} may be applied directly up to  minor changes of
notation, and thus we shall refer to \cite{ADKM} for details of the
proofs; and conditional estimates, which are different from those of
$X_{hi}$, but easier due to hypotheses of disjointness of supports, and for which we shall give the proofs.

 Recall that if ${\mathcal F}$ is a family of finite subsets of $\N$, then
$${\mathcal F}^{\prime}=\{A \cup B: A, B \in {\mathcal F}, A \cap
B=\emptyset\}.$$ Given $\varepsilon >0$ and $j=2,3,\ldots $, an
$(\varepsilon ,j)$-{\it basic special convex combination
($(\varepsilon ,j)$- basic s.c.c.) (relative to $X_{M(1),u})$} is a
vector of the form $\sum_{k\in F}a_ke_k$ such that: $F\in {\mathcal
M}_j,a_k\geq 0, \sum_{k\in F}a_k=1$, $\{a_k\}_{k\in F}$ is
decreasing, and, for every $G\in {\mathcal F}^{\prime
}_{t_j(k_{j-1}+1)}$, $\sum_{k\in G}a_k< \varepsilon $.

Given a  block sequence $(x_k)_{k\in {\bf N}}$ in $X_{u}$ and $j\geq 2$,
a convex combination
$\sum_{i=1}^na_ix_{k_i}$ is said to be an $(\varepsilon ,j)$-{\it
special convex combination} of $(x_k)_{k\in {\bf N}}$ ($(\varepsilon
,j)$-s.c.c), if there exist $l_1<l_2<\ldots <l_n$ such that $2<{\rm
supp}\ x_{k_1}\leq l_1<{\rm supp}\ x_{k_2}\leq l_2< \ldots <{\rm
supp}\ x_{k_n}\leq l_n$, and $\sum_{i=1}^na_ie_{l_i}$ is an
$(\varepsilon , j)$-basic s.c.c.
 An $(\varepsilon ,j)$-s.c.c. $\sum_{i=1}^n a_ix_{k_i}$
 is called {\it seminormalised} if  $\| x_{k_i}\|=1,\; i=1,\ldots ,n$ and
$$\|\sum_{i=1}^na_ix_{k_i}\|\geq\frac{1}{2}.$$

Rapidly increasing sequences and $(\varepsilon , j)$--R.I. special
convex combinations in $X_u$ are defined by \cite{ADKM} Definitions 2.8 and 2.16
respectively, with $\tilde{K}=L.$

Using the lower estimate for ${\mathcal M}_{2j}$-admissible families
in $X_u$ we get as in \cite{ADKM} Lemma 3.1.

\begin{lemme}\label{scc} For $\epsilon>0$, $j=1,2,\ldots$ and every
normalised block sequence $\{ x_k\}_{k=1}^{\infty }$ in $X_u$, there
exists a finite normalised block sequence $(y_s)_{s=1}^n$ of $(x_k)$
and coefficients $(a_s)_{s=1}^ n$ such that $\sum_{s=1}^na_sy_s$ is
a seminormalised $(\epsilon,2j)$--s.c.c..
\end{lemme}

The following definition is inspired from some of the hypotheses of
\cite{ADKM} Proposition 3.3.

\begin{defi}
 Let $j>100$. Suppose that
$\{ j_k\}_{k=1}^n$, $\{ y_k\}_{k=1}^n$, $\{ c_k^{\ast }\}_{k=1}^n$
and $\{b_k\}_{k=1}^n$ are such that

\medskip

{\rm (i)} There exists a rapidly increasing sequence
$$\{ x_{(k,i)}:\; k=1,\ldots ,n,\; i=1,\ldots ,n_k\} $$ with
$x_{(k,i)}<x_{(k,i+1)}<x_{(k+1,l)}$ for all $k<n$, $i<n_k$, $l\leq
n_{k+1},$ such that:

\noindent {\rm (a)} Each $x_{(k,i)}$ is a seminormalised
$(\frac{1}{m^4_{j_{(k,i)}}}, j_{(k,i)})$--s.c.c. where, for each
$k$, $2j_k+2<j_{(k,i)},\; i=1,\ldots n_k.$

\noindent {\rm (b)} Each $y_k$ is a $(\frac{1}{m^4_{2j_k}},2j_k)$--
 R.I.s.c.c. of  $\{ x_{(k,i)}\}_{i=1}^{n_k}$ of the form
$y_k=\sum _{i=1}^{n_k}b_{(k,i)}x_{(k,i)}.$

\noindent {\rm (c)} The sequence $\{ b_k\}_{k=1}^n$ is decreasing
and  $\sum _{k=1}^nb_ky_k$ is a $(\frac{1} {m^4_{2j+1}},
2j+1)$--s.c.c.

\medskip

{\rm (ii)} $c_k^{\ast }\in {\rm conv}_{\Q}L_{2j_k}$,  and $\max({\rm
supp}\ c_{k-1}^{\ast } \cup {\rm supp}\ y_{k-1}) < \min({\rm supp}\
c_k^* \cup {\rm supp}\ y_k)$, $\forall k$.

\medskip

{\rm (iii)} $j_1>2j+1$ and $2j_k=\Phi (c_1^{\ast },\ldots
,c_{k-1}^{\ast })$, $k=2,\ldots ,n$.

\medskip

Then $(j_k,y_k,c_k^*,b_k)_{k=1}^n$ is said to be a {\em
$j$-quadruple}.
\end{defi}

The following proposition is essential. It is the counterpart of
Lemma \ref{critical} for the space $X_u$.

\begin{prop}\label{criticalbis} Assume that $(j_k,y_k,c_k^*,b_k)_{k=1}^n$ is a
$j$-quadruple
  in $X_u$ such that
${\rm supp}\ c_k^* \cap {\rm supp}\ y_k=\emptyset$ for all
$k=1,\ldots,n$. Then
$$\norm{\sum_{k=1}^n b_km_{2j_k}y_k} \leq \frac{75}{m_{2j+1}^2}.$$
\end{prop}

\begin{proof}
  Our aim is to show that
for every $\varphi\in\cup_{i=1}^{\infty }{\mathcal B}_i$,
$$\varphi (\sum_{k=1}^n b_k m_{2j_k}y_k)\leq
\frac{75}{m_{2j+1}^2}.$$ The proof is given in several steps.

{\tt 1st Case}: $w(\varphi)=\frac{1}{m_{2j+1}}$. Then $\varphi$ has
the form $\varphi =\frac{1}{m_{2j+1}}(Ey^*_{1}+\cdots
+Ey^*_{k_2}+Ey^*_{k_2+1}+\cdots Ey^*_d)$ where $E$ is a subset of
$\N$ and where $y_k^* \in L_{2j_k}(c_k^*)\ \forall k \leq k_2$ and
$y_k^* \in L_{2j_k}(d_k^*)\ \forall k \geq k_2+1$, with $d_{k_2+1}^*
\neq c_{k_2+1}^*$ (this is similar to the form of such a functional
in $X_{hi}$ but with the integer $k_1$ defined there equal to $1$ in our case).

If $k \leq k_2$ then $c_s^*$ and therefore $y_s^*$ is disjointly
supported from $y_k$, so $Ey_s^*(y_k)=0$ for all $s$, and therefore
$\varphi(y_k)=0$. If $k=k_2+1$ then we simply have $|\varphi(y_k)|
\leq \norm{y_k} \leq 17m_{2j_k}^{-1}$, \cite{ADKM} Corollary 2.17.
Finally if $k>k_2+1$ then since $\Phi$ is 1-1, we have that
$j_{k_2+1} \neq j_k$ and for all $s=k_2+1,\ldots,d$, $d_s^*$ and
therefore $y_s^*$ belong to  ${\mathcal B}_{2t_s}$ with $t_s \neq
j_k$. It is then easy to check that we may reproduce the proof of
\cite{ADKM} Lemma 3.5, applied to $Ey_1^*,\ldots,Ey_d^*$, to obtain
the unconditional estimate
$$|\varphi(m_{2j_k}y_k)| \leq \frac{1}{m_{2j+1}^2}.$$
In particular instead of \cite{ADKM} Proposition 3.2, which is a
reformulation of \cite{ADKM} Corollary 2.17 for $X_{hi}$, we simply
use \cite{ADKM} Corollary 2.17 with $\tilde{K}=L$.

Summing up these estimates we obtain the desired result for the 1st
Case.

{\tt 2nd Case}: $w(\varphi )\leq \frac{1}{m_{2j+2}}.$ Then we get an
unconditional estimate for the evaluation of $\varphi (\sum
_{k=1}^{n} b_k m_{2j_k}y_k)$ directly, reproducing the short proof
of \cite{ADKM} Lemma 3.7, using again \cite{ADKM} Corollary 2.17
instead of \cite{ADKM} Proposition 3.2. Therefore

$$|\varphi (\sum_{k=1}^nb_km_{2j_k}y_k)|
\leq\frac{35}{m_{2j+2}} \leq \frac{35}{m_{2j+1}^2}.$$

{\tt 3rd Case}: $w(\varphi)>\frac{1}{m_{2j+1}}$.
 We have $y_k=\sum_{i=1}^{n_k}b_{(k,i)}x_{(k,i)}$
and the sequence $\{ x_{(k,i)},k=1,\ldots n,i=1,\ldots n_k\}$ is a
R.I.S. w.r.t. $L$. By \cite{ADKM} Proposition 2.9 there exist a
functional $\psi\in K^{\prime }$ (see the definition in \cite{ADKM}
p 71) and blocks of the basis $u_{(k,i)}$, $k=1,\ldots ,n$,
$i=1,\ldots ,n_k$ with
 ${\rm supp}\ u_{(k,i)}\subseteq {\rm supp}\ x_{(k,i)}$,
$\| u_k\|_{\ell_1}\leq 16$ and such that
$$|\varphi (\sum_{k=1}^nb_km_{2j_k}
(\sum_{i=1}^{n_k}b_{(k,i)}x_{(k,i)}))|\leq m_{2j_1}b_1b_{(1,1)}+\psi
(\sum_{k=1}^nb_k
m_{2j_k}(\sum_{i=1}^{k_n}b_{(k,i)}u_{(k,i)}))+\frac{1}{m_{2j+2}^2}$$
$$\leq\psi (\sum_{k=1}^nb_km_{2j_k}(\sum_{i=1}^{k_n}
b_{(k,i)}u_{(k,i)}))+\frac{1}{m_{2j+2}}.$$ Therefore  it suffices to
estimate
$$\psi(\sum_{k=1}^nb_km_{2j_k}(\sum_{i=1}^{n_k}
b_{(k,i)}u_{(k,i)})).$$

In \cite{ADKM}  $\psi$ is decomposed as $\psi_1+\psi_2$ and
different estimates are applied to $\psi_1$ and $\psi_2$. Our case
is easier as we may simply assume that $\psi_1=0$ and
$\psi_2=\psi$. We shall therefore refer to some arguments of
\cite{ADKM} concerning some $\psi_2$ keeping in mind that  $\psi_2=\psi$.

Let $D_1^k,\ldots,D_4^k$ be defined as in \cite{ADKM} Lemma 3.11
(a). Then as in \cite{ADKM}, $$\bigcup_{p=1}^4D_p^k=\bigcup _{i=1}^{n_k}
{\rm supp}\ u_{(k,i)}\cap {\rm supp}\ \psi.$$ The proof that
$$\psi|_{\bigcup_kD_2^k}(\sum_kb_km_{2j_k}(\sum_ib_{(k,i)}u_{(k,i)}))
\leq\frac{1}{m_{2j+2}}, \leqno (1)$$
$$\psi|_{\bigcup_kD_3^k}(\sum_kb_km_{2j_k}(\sum_ib_{(k,i)}u_{(k,i)}))
\leq\frac{16}{m_{2j+2}}, \leqno (2)$$ and
$$\psi|_{\bigcup_kD_1^k}(\sum_kb_km_{2j_k}(\sum_ib_{(k,i)}
u_{(k,i)}))\leq\frac{1}{m_{2j+2}}. \leqno (3)$$ may be easily
reproduced from \cite{ADKM} Lemma 3.11. The case of $D_4^k$ is
slightly different from \cite{ADKM} and therefore we give more
details. We claim

\

\noindent{\em Claim:} Let $D=\bigcup_k D_4^k$. Then
$$\psi|_{D}(\sum_kb_km_{2j_k}(\sum_ib_{(k,i)}u_{(k,i)}))
\leq\frac{64}{m_{2j+2}}, \leqno (4)$$

Once the claim is proved it follows by adding the estimates that the
3rd Case is proved, and this concludes the proof of the Proposition.

\

\noindent{\em Proof of the claim:} Recall that $D_4^k$ is defined by
$$D_4^k=\{
m\in\bigcup_{i=1}^{n_k}A_{(k,i)} : {\rm for\ all}\ f\in\bigcup_sK^s(\psi)\
{\rm with}\; m\in {\rm supp}f, w(f) \geq \frac{1}{m_{2j_k}}\;{\rm
and}$$
$${\rm there}\;{\rm
exists}\; f\in\bigcup_sK^s(\psi)\;{\rm with}\; m\in {\rm supp}f,
 w(f)=\frac{1}{m_{2j_k}}{\rm and}$$
$${\rm for}\;{\rm
every}\;g\in\bigcup_sK^s(\psi)\; {\rm with}\; {\rm supp}\ f \subset
{\rm supp}\ g {\rm \ strictly}, w(g)\geq\frac{1}{m_{2j+1}}\}.$$

 For every $k=1,\ldots ,n$, $i=1,\ldots ,n_k$
and every $m\in {\rm supp}\ u_{(k,i)}\cap D_4^k,$ there exists a
unique functional $f^{(k,i,m)}\in \bigcup _sK^s(\psi)$ with $m\in {\rm
supp}\ f$, $w(f)=\frac{1}{m_{2j_k}}$ and such that, for all $g\in
\bigcup_sK^s(\psi)$ with ${\rm supp}\ f \subseteq {\rm supp}\ g$
strictly, $w(g)\geq \frac{1}{m_{2j+1}}.$ By definition, for $k\neq
p$ and $i=1,\ldots ,n_k$, $m\in {\rm supp}\ u_{(k,i)}$, we have
${\rm supp}f^{(k,i,m)}\cap D^p_4=\emptyset.$ Also, if
$f^{(k,i,m)}\neq f^{(k,r,n)},$ then ${\rm supp}\ f^{(k,i,m)}\cap
{\rm supp}\ f^{(k,r,n)}=\emptyset.$

For each $k=1,\ldots ,n$, let $\{ f^{k,t}\}_{t=1}^{r_k}\subseteq \bigcup
K^s(\varphi ) $ be a selection of mutually disjoint such functionals
with $D^k_4=\bigcup _{t=1}^{r_k} {\rm supp}\ f^{k,t}.$ For each such
functional $f^{k,t}$, we set
$$a_{f^{k,t}}=\sum_{i=1}^{n_k}b_{(k,i)}\sum_{m\in {\rm supp}\ f^{k,t}} a_m.$$
Then,
$$f^{k,t}(b_km_{2j_k}(\sum_ib_{(k,i)}u_{(k,i)}))\leq
b_ka_{f^{k,t}}.\leqno (5)$$ We define as in \cite{ADKM} a functional
$g\in K^{\prime }$ with $|g|_{2j}^{\ast }\leq 1$ (see definition
\cite{ADKM} p 71), and blocks $u_k$ of the basis so that $\|
u_k\|_{\ell_1}\leq 16$, ${\rm supp}\ u_k\subseteq \bigcup_i{\rm supp}\
u_{(k,i)}$ and
$$\psi|_{D_4}(\sum_kb_km_{2j_k}(\sum_ib_{(k,i)}u_{(k,i)}))
\leq g(2\sum_kb_ku_k),$$ \noindent hence by \cite{ADKM} Lemma 2.4(b)
we shall have the result.

For $f=\frac{1}{m_q}\sum_{p=1}^df_p\in\bigcup_sK^s(\psi|_{D_4})$ we set
$$J=\{ 1\leq p\leq d: f_p=f^{k,t}\;{\rm for}\;{\rm some}\; k=1\ldots ,n, \;
t=1,\ldots , r_k\},$$
$$T=\{ 1\leq p\leq d:\;{\rm there}\;{\rm exists}\;f^{k,t}\;{\rm with}\;
{\rm supp}f^{k,t} \subseteq {\rm supp}f_p \; {\rm strictly}\}.$$

For every $f\in\bigcup_sK^s(\psi|_{D_4})$ we shall define by induction
a functional $g_f$, by $g_f=0$ when $J\cup T=\emptyset$,
 while if $J\cup T\neq\emptyset $ we shall construct $g_f$ with the following
properties. Let $D_f= \bigcup_{p\in J\cup T}{\rm supp}f_p$ and
$u_k=\sum a_{f^{k,t}}e_{f^{k,t}}$, where $e_{f^{k,t}}=e_{\min{\rm
supp} f^{k,t}}$, then:

(a) ${\rm supp}\ g_f\subseteq {\rm supp}\ f$.

(b) $g_f\in K^{\prime }$ and $w(g_f)\geq w(f)$,

(c) $f|_{D_f}(\sum_kb_km_{2j_k}(\sum_ib_{(k,i)}u_{(k,i)})) \leq
g_f(2\sum_kb_ku_k)$.

\

\noindent Let $s>0$ and suppose that $g_f$ have been defined for all
$f\in\bigcup_{t=0}^{s-1}K^t(\psi|_{D_4})$ and let
$f=\frac{1}{m_q}(f_1+\ldots +f_d)\in K^s(\psi|_{D_4})\backslash
K^{s-1}(\psi|_{D_4})$ where the family $(f_p)_{p=1}^d$ is ${\mathcal
M}_q$-admissible if $q>1$, or ${\mathcal S}$-allowable if $q=1$. The
proofs of case (i) ($1/m_q=1/m_{2j_k}$ for some $k \leq n$) and case
(ii) ($1/m_q>1/m_{2j+1}$) are identical with \cite{ADKM} p 106.
Assume therefore that case (iii) holds, i.e., $1/m_q=1/m_{2j+1}$. For
the same reasons as in \cite{ADKM} we have that $T=\emptyset$.

Summing up we assume that $f \in K^s(\psi|_{D_4})\backslash
K^{s-1}(\psi|_{D_4})$ is of the form
$$f=\frac{1}{m_{2j+1}}\sum_{p=1}^d
  f_p=\frac{1}{m_{2j+1}}(Ey_1^*+\ldots+Ey_{k_2}^*+Ey_{k_2+1}^*+\ldots+Ey_d^*),$$
where $(y_i^*)_i$ is associated to
$(c_1^*,\ldots,c_{k_2}^*,d_{k_2+1}^*,\ldots)$ with $d_{k_2+1}^* \neq
c_{k_2+1}^*$, that $T=\emptyset$ and $J \neq \emptyset$, and it only
remains to define $g_f$ satisfying (a)(b)(c).

Now by the proof of \cite{ADKM} Proposition 2.9,
$\psi=\psi_{\varphi}$ was defined through the analysis of $\varphi$,
in particular by \cite{ADKM} Remark 2.19 (a),
$$\psi=\frac{1}{m_{2j+1}}\sum_{k \in I}\psi_{Ey_k^*}$$
for some subset $I$ of $\{1,\ldots,d\}$. Furthermore, for $l \in I$,
$l \leq k_2$ and $1 \leq k \leq d$, ${\rm supp}\ Ey_l^* \cap {\rm
supp}\ x_k=\emptyset$, therefore there is no functional in a family
of type I and II w.r.t. $\overline{x_k}$ of support included in
${\rm supp}\ Ey_l^*$ (see \cite{ADKM} Definition 2.11 p 77). This
implies that $D_{Ey_l^*}=\emptyset$ (\cite{ADKM} Definition p 85),
and therefore that $\psi_{Ey_l^*}=0$ (\cite{ADKM} bottom of p 85).

For $l \in I$, $l>k_2+1$, then since $\Phi$ is $1-1$,
$w(Ey_l^*)=w(Ed_l^*) \neq 1/m_{2j_k} \forall k$. Therefore
$w(\psi_{Ey_l^*}) \neq 1/m_{2j_k} \forall k$,
 \cite{ADKM} Remark 2.19 (a). Then by the definition of $D_4^k$,
${\rm supp}\ \psi_{Ey_l^*} \cap D_4^k=\emptyset$ for all $k$.

 Finally this
means that $\psi_{|D_4}=\frac{1}{m_{2j+1}}\psi_{Ey^*_{k_2+1}|D_4}$
and $J=\{k_2+1\}$, $D_f={\rm supp}\ f_{k_2+1}$. Write then
$f_{k_2+1}=f^{k_0,t}$ and set $g_f=\frac{1}{2}e^*_{f_{k_2+1}}$,
therefore (a)(b) are trivially verified. It only remains to check
(c). But by (5),
$$f|_{D_f}(\sum_kb_km_{2j_k}(\sum_ib_{(k,i)}u_{(k,i)}))\leq
b_{k_0} a_{f_{k_2+1}}$$
$$=b_{k_0} a_{f_{k_2+1}} e^*_{f_{k_2+1}}(e_{f_{k_2+1}})
=g_f(2b_{k_0} a_{f_{k_2+1}}e_{f_{k_2+1}})$$
$$=g_f(2\sum_t b_{k_0} a^{f_{k,t}}e_{f^{k,t}})
=g_f(2\sum_kb_ku_k).$$ So (c) is proved. Therefore $g_f$ is defined
for each $f$ by induction, and the Claim is verified. This concludes
the proof of the Proposition.
\end{proof}

\begin{prop} The space $X_u$ is of type (3). \end{prop}

\begin{proof} Assume towards a contradiction that $T$ is an isomorphism from
  some block-subspace $[x_n]$ of $X_u$  into the subspace $[e_i, i \notin \bigcup_n
{\rm supp}\ x_n]$. We may assume that $\max({\rm supp}\ x_n,{\rm
supp}\ Tx_n) < \min({\rm supp}\ x_{n+1},{\rm supp}\ Tx_{n+1})$ and
$\min{\rm supp}\ x_n<\min{\rm supp}\ Tx_n$ for each $n$, and by
Lemma \ref{scc}, that each $x_n$ is a $(\frac{1}{m_{2n}^4},2n)$
R.I.s.c.c. (\cite{ADKM} Definition 2.16). We may write
$$x_n=\sum_{t=1}^{p_n} a_{n,t}x_{n,t}$$ where $(x_{n,1},\ldots,x_{n,p_n})$ is
${\mathcal M}_{2n}$-admissible. Let for each $n,t$, $x_{n,t}^* \in
L$ be such that ${\rm supp}\ x_{n,t}^* \subseteq {\rm supp}\ Tx_{n,t}$
and such that
$$x_{n,t}^*(Tx_{n,t}) \geq \frac{1}{2}\norm{Tx_{n,t}} \geq
\frac{1}{4\norm{T^{-1}}},$$ and let
$x_n^*=\frac{1}{m_{2n}}(x_{n,1}^*+\ldots+x_{n,p_n}^*) \in L_{2n}$.
Note that ${\rm supp}\ x_n^* \cap {\rm supp}\ x_n=\emptyset$ and
that
$$x_n^*(Tx_n) \geq \frac{1}{m_{2n}}\sum_{t=1}^{p_n}
\frac{a_{n,t}}{4\norm{T^{-1}}}= (4\norm{T^{-1}}m_{2n})^{-1}.$$ We
may therefore for any $j>100$ construct a $j$-quadruple
 $(j_k,y_k,c_k^*,b_k)_{k=1}^n$ satisfying the hypotheses of Proposition
 \ref{criticalbis} and such that $y_k \in [x_i]_i$ and $c_k^*(Ty_k) \geq
 (4\norm{T^{-1}}m_{2j_k})^{-1}$ for each $k$ (note that we may assume
that $c_k^* \in L_{j_{2k}}$ for each $k$). From Proposition
 \ref{criticalbis} we deduce
$$\norm{\sum_{k=1}^n b_k m_{2j_k}y_k} \leq \frac{75}{m_{2j+1}^2}.$$
On the other hand $\psi=\frac{1}{m_{2j+1}}\sum_{k=1}^n c_k^*$
belongs to $L$ therefore
$$\norm{T(\sum_{k=1}^n b_k m_{2j_k}y_k)} \geq
\psi(\sum_{k=1}^n b_k m_{2j_k}Ty_k) \geq
\frac{1}{4\norm{T^{-1}}m_{2j+1}}.$$ We deduce finally that
$$m_{2j+1} \leq 300 \norm{T}\norm{T^{-1}},$$
which contradicts the boundedness of $T$.
\end{proof}

\subsection{A strongly asymptotically $\ell_{\infty}$ space tight by support}
Since the canonical basis of $X_u$ is tight and unconditional, it
follows that $X_u$ is reflexive. In particular this implies that the
dual basis of the canonical basis of $X_u$ is a strongly
asymptotically $\ell_{\infty}$ basis of $X_u^*$. It remains to prove
that this basis is tight with support.

It is easy to prove by duality that for any ${\mathcal
M}_{2j}$-admissible sequence of functionals $f_1,\ldots,f_n$ in
$X_u^*$, we have the upper estimate
$$\norm{\sum_i f_i} \leq m_{2j}\sup_i \norm{f_i}.$$
We use this observation to prove a lemma about the existence of
s.c.c. normed by functionals belonging to an arbitrary subspace of
$X_u^*$. The proof is standard except that estimates have to be
taken in $X_u^*$ instead of $X_u$.

\begin{lemme}\label{sccbis} For $\epsilon>0$, $j=1,2,\ldots$ and every
normalised block sequence $\{f_k\}_{k=1}^{\infty }$ in $X_u^*$,
there exists a normalised functional $f \in [f_k]$ and a
seminormalised $(\epsilon,2j)$--s.c.c. $x$ in $X_u$ such that ${\rm
supp}\ f \subseteq {\rm supp}\ x$ and $f(x) \geq 1/2$.
\end{lemme}

\begin{proof}
For each $k$ let $y_k$ be normalised such that ${\rm supp}\ y_k={\rm
supp}\ f_k$ and $f_k(y_k)=1$. Recall that the integers $k_n$ and
$t_n$ are defined by $k_1=1$, $2^{t_n}\geq m_n^2$ and
$k_n=t_n(k_{n-1}+1)+1$, and that ${\mathcal M}_j={\mathcal F}_{k_j}$
for all $j$.

 Applying Lemma \ref{scc}
we find a successive sequence of $(\epsilon,2j)$--s.c.c. of $(y_k)$
of the form $(\sum_{i \in I_k}a_i y_i)_k$ with $\{f_i, i \in I_k\}$
${\mathcal F}_{k_{2j-1}+1}$-admissible. If $\norm{\sum_{i \in
I_k}f_i} \leq 2$ for some $k$, we are done, for then
$$(\sum_{i \in I_k}f_i)(\sum_{i \in I_k}a_i y_i) \geq \frac{1}{2}\norm{\sum_{i
\in I_k}f_i}.$$ So assume $\norm{\sum_{i \in I_k}f_i}>2$ for
all $k$, apply the same procedure to the sequence
$f_k^1=\norm{\sum_{i \in I_k}f_i}^{-1}\sum_{i \in I_k}f_i$, and
obtain a successive sequence of $(\epsilon,2j)$--s.c.c. of the
sequence $(y_k^1)_k$ associated to $(f_k^1)_k$, of the form
$(\sum_{i \in I_k^1}a_i^1 y_i^1)_k$, with $\{f_l: {\rm supp}\ f_l
\subseteq \sum_{i \in I_k^1}f_i^1\}$ a ${\mathcal
F}_{k_{2j-1}+1}[{\mathcal F}_{k_{2j-1}+1}]$-admissible, and
therefore ${\mathcal M}_{2j}$-admissible set. Then we are done
unless $\norm{\sum_{i \in I_k^1}f_i^1}> 2$ for all $k$, in which
case we set
$$f_k^2=\norm{\sum_{j \in I_k^1}f_j^1}^{-1}
\sum_{j \in I_k^1}f_j^1$$ and observe by the upper estimate in
$X_u^*$ that
$$1=\norm{f_k^2}=\norm{\sum_{j \in I_k^1}\sum_{i \in I_j}\norm{\sum_{j \in
I_k^1}f_j^1}^{-1} \norm{\sum_{i \in I_j}f_i}^{-1} f_i} \leq
m_{2j}/4.$$ Repeating this procedure we claim that we are done in at
most $t_{2j}$ steps. Otherwise we obtain that the set
$$A=\{f_l:
{\rm supp}\ f_l \subseteq \sum_{i \in
I_k^{t_{2j-1}}}f_i^{t_{2j-1}}\}$$ is ${\mathcal M}_{2j}$-admissible.
Since $f_k^{t_{2j}}=\sum_{f_l \in A}\alpha_l f_l$, where the
normalising factor $\alpha_l$ is less than $(1/2)^{t_{2j}}$ for each
$l$, we deduce from the upper estimate that
$$1=\norm{f_k^{t_{2j}}} \leq 2^{-t_{2j}}m_{2j},$$
a contradiction by definition of the integers $t_i$'s. \end{proof}

\

To prove the last proposition of this section we need to make two
observations. First if $(f_1,\ldots,f_n) \in {\rm conv}_{\Q}L$ is
${\mathcal
  M}_{2j}$-admissible, then
$\frac{1}{m_{2j}}\sum_{k=1}^n f_k \in {\rm conv}_{\Q}L_{2j}.$
 Indeed  using the stability of $L$ under projections onto subsets of $\N$ we
may easily find convex rational coefficients $\lambda_i$
  such that each $f_k$ is of the form
$$f_k=\sum_i \lambda_i f_i^k,\ f_i^k \in L,\ {\rm supp}\ f_i^k
 \subseteq {\rm supp}\ f_k\ \forall i.$$
Then $\frac{1}{m_{2j}}\sum_{k=1}^n f_k=\sum_i \lambda_i
 (\frac{1}{m_{2j}}\sum_{k=1}^n f_i^k)$ and each
$\frac{1}{m_{2j}}\sum_{k=1}^n f_i^k$ belongs to $L_{2j}$.

Likewise if $\psi=\frac{1}{m_{2j+1}}(c_1^*+\ldots+c_d^*)$, $k>2j+1$,
$c_1^* \in {\rm conv}_{\Q} L_{2k}$ and $c_l^* \in {\rm conv}_{\Q}
L_{\Phi(c_1^*,\ldots,c_{l-1}^*)}\ \forall l \geq 2$, then $\psi \in
{\rm conv}_{\Q} L$. Indeed as above we may write
$$\psi=\sum_i \lambda_i (\frac{1}{m_{2j+1}}\sum_{l=1}^d f_i^l),\
f_i^1 \in L_{2k}, f_i^l \in L_{\Phi(c_1^*,\ldots,c_{i-1}^*)}(c_i^*)\
\forall l \geq 2,$$ and each $\frac{1}{m_{2j+1}}\sum_{l=1}^d f_i^l$
belongs to $L_{2j+1}^{\prime n+1} \subseteq L$.

\begin{prop} The space $X_u^*$ is of type (3).
\end{prop}

\begin{proof} Assume towards a contradiction that $T$ is an isomorphism from
  some block-subspace $[f_n]$ of $X_u^*$  into the subspace $[e_i^*, i \notin
\cup_n {\rm supp}\ f_n]$. We may assume that $\max({\rm supp}\
f_n,{\rm supp}\ Tf_n) < \min({\rm supp}\ f_{n+1},{\rm supp}\
Tf_{n+1})$ and $\min{\rm supp}\ Tf_n<\min{\rm supp}\ f_n$ for each
$n$. Since the closed unit ball of $X_u^*$ is equal to
$\overline{{\rm conv}_{\Q}L}$ we may also assume that $f_n \in {\rm
conv}_{\Q}L$ for each $n$. Applying Lemma \ref{sccbis}, we may also
suppose that each $f_n$ is associated to a $(\frac{1}{m_{2n}^4},2n)$
s.c.c. $x_n$ with $Tf_n(x_n) \geq 1/3$ and ${\rm supp}\ x_n \subset
{\rm supp}\ Tf_n$, and we shall also assume that $\norm{Tf_n}=1$ for
each $n$. Build then for each $k$ a $(\frac{1}{m_{2k}^4},2k)$
R.I.s.c.c. $y_k=\sum_{n \in A_k}a_n x_n$ such that $(Tf_n)_{n \in
A_k}$ and therefore $(f_n)_{n \in A_k}$ is ${\mathcal
M}_{2k}$-admissible. Then note that by the first observation
before this proposition,
$$c_k^*:=m_{2k}^{-1}\sum_{n \in A_k}f_n \in {\rm conv}_{\Q}L_{2k},$$
and  observe that ${\rm supp}\ c_k^* \cap {\rm supp}\
y_k=\emptyset$ and that $Tc_k^*(y_k) \geq (3m_{2k})^{-1}$.

We may therefore for any $j>100$ construct a $j$-quadruple
 $(j_k,y_k,c_k^*,b_k)_{k=1}^n$ satisfying the hypotheses of Proposition
 \ref{criticalbis} and such that $c_k^* \in [f_i]_i$ and $Tc_k^*(y_k) \geq
 (3m_{2j_k})^{-1}$ for each $k$. From
Proposition
 \ref{criticalbis} we deduce
$$\norm{\sum_{k=1}^n b_k m_{2j_k}y_k} \leq \frac{75}{m_{2j+1}^2}.$$
Therefore
$$\norm{\sum_{k=1}^d Tc_k^*} \geq
\frac{\sum_{k=1}^d b_k m_{2j_k}Tc_k^*(y_k)}{\norm{\sum_{k=1}^n b_k
m_{2j_k}y_k}} \geq \frac{m_{2j+1}^2}{225},$$ but on the other hand
$$\norm{\sum_{k=1}^d c_k^*} \leq m_{2j+1}$$
since by the second observation the functional
$m_{2j+1}^{-1}\sum_{k=1}^d c_k^*$ belongs to ${\rm conv}_{\Q}L$. We
deduce finally that
$$m_{2j+1} \leq 225 \norm{T},$$
which contradicts the boundedness of $T$.
\end{proof}

\section{Problems and comments}

Obviously the general question one is compelled to ask is whether it is possible to find an example for each of the classes or subclasses appearing in the chart of Theorem \ref{final}. 
However we wish to be more specific here and concentrate on the classes which either seem particularly interesting, or easier to study, or which are related to one of the spaces considered in this paper.

\

Let us first observe that the examples of locally minimal, tight spaces produced so far could
be said to be so for trivial reasons: since they hereditarily contain
$\ell_{\infty}^n$'s uniformly, any Banach space is crudely finitely
representable in any of their subspaces.
It remains open whether there exist other examples. Observing that a locally minimal and tight space cannot be strongly asymptotically $\ell_p, 1 \leq p<+\infty$, by one of the implications in the  diagram before Theorem \ref{final}, and up to the 6th dichotomy, the problem may be summed up as:

\begin{prob} Find
 a tight, locally minimal, uniformly inhomogeneous Banach space
which does not contain $\ell_{\infty}^n$'s uniformly, or equivalently, which has finite
cotype. 
\end{prob}

It also unknown whether tightness by range and tightness with constants are the only possible forms of tightness, up to passing to subspaces. Equivalently, using the 4th and 5th dichotomy:

\begin{prob} Find a tight Banach space which is sequentially and locally minimal.
\end{prob}

Or, since such a space would have to be of type (2), (5b), (5d): 

\begin{prob} \

 \begin{itemize}
\item[(a)] Find a HI space which is sequentially minimal.
\item[(b)] Find a space of type (5b).
\item[(c)]  Find a space of type (5d). Is the dual of some modified mixed Tsirelson's space such a space?
\end{itemize}
\end{prob}

For the next problem, we observe that the only  known examples of spaces tight with constants are strongly asymptotic $\ell_p$ spaces not containing $\ell_p$, where $1 \leq p <+\infty$.

\begin{prob} Find a space tight with constants and uniformly inhomogeneous. \end{prob}

More specifically, listing two  subclasses for which we have a possible candidate:

\begin{prob}
\

\begin{itemize} 
\item[(a)] Find a space of type (1a). Is $G$ or one of its subspaces such a space?
\item[(b)] Find a space of type (3a). Is $G_u$ or one of its subspaces such a space?
\end{itemize}
\end{prob}

\

Recently, S. Argyros, K. Beanland and T. Raikoftsalis \cite{ABR,ABR2} constructed an example $X_{abr}$ with a basis which is strongly asymptotically $\ell_2$ and therefore weak Hilbert, yet every operator is a strictly singular perturbation of a diagonal map, and no disjointly supported subspaces are isomorphic. In our language, $X_{abr}$ is therefore a new space of type (3c), which we include in our chart.

We conclude by mentioning the very recent and remarkable result of  S. Argyros and R. Haydon solving the scalar plus compact problem \cite{AH}: there exists a HI space which is a predual of $\ell_1$ and on which every operator is a compact perturbation of a multiple of the identity. To our knowledge nothing is known about the exact position of this space in the chart of Theorem \ref{final}.

\begin{prob} Find whether Argyros-Haydon's space is of type (1) or of type (2). \end{prob}

\

\begin{flushleft}

{\em Address of V. Ferenczi:}\\

Departamento de Matem\'atica,\\

Instituto de Matem\'atica e Estat\' \i stica,\\

Universidade de S\~ao Paulo,\\

rua do Mat\~ao, 1010, \\

05508-090 S\~ao Paulo, SP,\\

Brazil.\\
\texttt{ferenczi@ime.usp.br}
\end{flushleft}

\

\begin{flushleft}
{\em Address of C. Rosendal:}\\
Department of Mathematics, Statistics, and Computer Science\\
University of Illinois at Chicago,\\
851 S. Morgan Street,\\
Chicago, IL 60607-7045,\\
USA.\\
\texttt{rosendal@math.uic.edu}
\end{flushleft}

\end{document}